\documentclass[12pt]{amsart}
\usepackage{amsmath, amssymb}
 \usepackage{epsfig}

\theoremstyle{plain}
\usepackage{amsmath}
\usepackage{amssymb}
\usepackage{amsfonts}
\usepackage{xcolor}
\newtheorem{theorem}{Theorem}

\newtheorem{lemma}{Lemma}
\newtheorem{remark}{Remark}

\numberwithin{equation}{section}

\newcommand{\ignore}[1]{}{}

\def \< {\langle}
\def \> {\rangle}
\def \^ {\widehat}

\newcommand{\bbA}{{\bf A}}

\newcommand{\bbB}{{\bf B}}
\newcommand{\bbC}{{\bf C}}

\newcommand{\bbD}{{\bf D}}

\newcommand{\bbF}{{\bf F}}

\newcommand{\bbG}{{\bf G}}
\newcommand{\bbH}{{\bf H}}

\newcommand{\bbI}{{\bf I}}

\newcommand{\bbs}{{\bf s}}

\newcommand{\bbT}{{\bf T}}

\newcommand{\bbX}{{\bf X}}

\newcommand{\bbx}{{\bf x}}

\newcommand{\beq}{\begin{equation}}
\newcommand{\eeq}{\end{equation}}
\newcommand{\bqa}{\begin{eqnarray}}
\newcommand{\eqa}{\end{eqnarray}}
\newcommand{\bqn}{\begin{eqnarray*}}
\newcommand{\eqn}{\end{eqnarray*}}
\newcommand{\non}{\nonumber \\}
\newcommand{\bdes}{\begin{description}}
\newcommand{\edes}{\end{description}}

\def\underwiggle 1{\ifmmode\setbox\TempBox=\hbox{$ 1$}\else\setbox\TempBox=\hbox{1}\fi
\setbox\TempBoxA=\hbox to \wd\TempBox{\hss\char'176\hss}
\rlap{\copy\TempBox}\smash{\lower9pt\hbox{\copy\TempBoxA}} }
\begin{document}

\title[ CLT of nonparametric estimate of density functions]{Central limit theorem of nonparametric estimate of spectral density functions of sample covariance matrices}

\author{ Guangming Pan, \  \ Qi-Man Shao, \ \ Wang Zhou \\
}
\thanks{ G.M. Pan was partially supported by a grant M58110052 at the Nanyang Technological University;
Q.M. Shao was partially supported by Hong Kong RGC CERG 602608 and 603710;   W. Zhou was partially supported by a grant
R-155-000-106-112 at the National University of Singapore}

\address{Division of Mathematical Sciences, School of Physical and Mathematical Sciences, Nanyang Technological University, Singapore 637371}
\email{gmpan@ntu.edu.sg}

\address{Department of Mathematics, Hong Kong University of Science and Technology, Clear Water Bay, Kowloon, Hong Kong}
\email{maqmshao@ust.hk}

\address{Department of Statistics and Applied Probability, National University of
 Singapore, Singapore 117546}
\email{stazw@nus.edu.sg} 

\maketitle

\begin{abstract}
A consistent kernel estimator of  the limiting spectral distribution of general sample covariance matrices
was introduced in Jing, Pan, Shao and Zhou (2010). The central limit theorem of the kernel estimator
 is proved in this paper.
 \end{abstract}

\def\theequation{\thesection.\arabic{equation}}

\section{Introduction}

\setcounter {equation}{0}

Spectral analysis of sample
 covariance matrices plays a  very important role in multivariate statistical
 inference since many test statistics are defined by its eigenvalues or functionals.
Let  $\bbX=(X_{ij})_{p\times n}$ be  independent and
identically distributed (i.i.d.) real-valued random variables and $\bbT$ be
a $p\times p$ non-random Hermitian non-negative definite matrix with $(\bbT^{1/2})^2=\bbT$.
Define the sample covariance matrix  by
$$
\bbA_n=\frac{1}{n}\bbT^{1/2}\bbX_n\bbX_n^T\bbT^{1/2}
$$
and its  empirical spectral distribution
$F_n^{\bbA}$ by
$$
F^{\bbA_n}(x)=\frac{1}{p}\sum\limits_{k=1}^pI(\lambda_k\leq x),
$$
where $\lambda_k,k=1,\cdots,p$ denote the eigenvalues of $\bbA_n$.
Instead of $\bbA_n$ we also consider
$$
\bbB_n=\frac{1}{n}\bbX_n^T\bbT\bbX_n,
$$
because the eigenvalues of $\bbA_n$ and $\bbB_n$ differ by $|n-p|$
zero eigenvalues. Suppose the ratio of the dimension and sample size
$c_n=p/n$ tends to a positive constant $c$ as $n\to \infty$. When
$F^{\bbT}$ converges weakly to a distribution $H$,
 it is
proved in Marcenko and Pastur \cite{MP}, Yin \cite{y1} and
Silverstein \cite{s3} that, with probability one, $F^{\bbB_n}(x)$
converges in distribution to an MP type distribution function
$\underline{F}^{c,H}(x)$
 whose Stieltjes transform $\underline m(z)=m_{\underline{F}^{c,H}}(z)$ is, for each $z\in \mathcal{C}^+=\{z\in \mathcal{C}: \Im z>0\}$, the unique solution to
the equation \begin{equation}\label{a3}
\underline{m}=-\left(z-c\int\frac{tdH(t)}{1+t\underline{m}}\right)^{-1}.
\end{equation}
Here the Stieltjes transform $m_F(z)$ for any probability
distribution function $F(x)$ is given by
\begin{equation}\label{b5*}
m_F(z)=\int\frac{1}{x-z}dF(x),\ \ z\in\mathcal{C}^+.
\end{equation}
Note that from (\ref{a3}) $\underline m(z)$ has an inverse
\begin{equation}\label{f51}
z=-\frac{1}{\underline{m}}+c\int\frac{t}{1+t\underline{m}}dH(t).
\end{equation}

Bai and Silverstein \cite{b2} established a far reaching central limit theorem (CLT) for the eigenvalues of $\bbA_n$, which makes  possible the hypothesis
testing of linear spectral statistics of sample covariance matrices indexed by analytic functions.
 Pan
and Zhou \cite{PZ} relaxed some restriction on the fourth moment of the
underlying random variables. Lytova and Marcenko \cite{LM} and Bai,
Wang and Zhou \cite{b3} further, respectively, extended Bai and Silverstein's theorem from the analytic
test function to the one having fourth derivative when $\bbT$
is the identity matrix. However, the limiting spectral distribution  $\underline{F}^{c,H}$
is usually unknown for general $T$. It is also not clear if there is any CLT
about $\big(F^{\bbB_n}(x)-\underline{F}^{c,H}(x)\big)$,
equivalently $\big(F_n^{\bbA_n}(x)-F^{c,H}(x)\big)$  even in the normal population,
 here $F^{c,H}(x)$ is the limiting distribution of $F_n^{\bbA_n}(x)$. How can
 one make inference for $f_{c,H}(x)$ or $F^{c,H}(x)$
 based on $F_n^{\bbA_n}(x)$ or $F^{\bbB_n}(x)$ without establishing
  CLTs?


Motivated by the ``smoothing" ideas, Jing, Pan, Shao and Zhou \cite{ker}
proposed the following  kernel estimator of the density
function of $\underline{F}^{c,H}(x)$ as
\begin{equation}
f_n(x)=\frac{1}{ph}\sum\limits_{i=1}^pK(\frac{x-\lambda_i}{h})=\frac{1}{h}\int
K(\frac{x-y}{h})dF_n^{\bbA_n}(y)\label{a1},
\end{equation}
where $h$ is the bandwidth.  It was proved
that $f_n(x)$ is a consistent estimator of $f_{c,H}(x)$ under some
regularity conditions.

The main aim of this paper is to establish a CLT for
$f_n(x)$. This provides an approach to making inference on the MP type distribution functions.
To this end, we first list some technical conditions on the
kernel function.

Suppose that the kernel function $K(x)$ satisfies
\begin{equation}\label{a25}
\lim\limits_{|x|\rightarrow\infty}|xK(x)|=\lim\limits_{|x|\rightarrow\infty}|xK'(x)|=0,\
\end{equation}
 \begin{equation} \label{a26} \int K(x)dx=1, \ \ \int
|xK'(x)|dx<\infty,  \ \ \int |K''(x)|dx<\infty.
\end{equation}
and
\begin{equation} \label{a27}
\int xK(x)dx=0,  \int x^2|K(x)|dx<\infty.
\end{equation}
Let $z=u+iv$ with $v$ being in a bounded interval, say $[-v_0,v_0]$ with $v_0>0$. Suppose that
\begin{equation}\label{f63}
\int^{+\infty}_{-\infty}|K^{(j)}(z)|du<\infty, \quad j=0,1,2,
\end{equation}
uniformly in $v\in [-v_0,v_0]$, where $K^{(j)}(z)$ denotes the $k$-th derivative of $K(z)$.

Some assumptions on $H_n(t):=F^{\bbT}$, are also needed.  Introduce the interval
\begin{equation}\label{a2}
\left[\lambda_{\min}(\bbT)(1-\sqrt{c_n})^2,\lambda_{\max}(\bbT)(1+\sqrt{c_n})^2\right].
\end{equation}
Denote the right and left end points of the above interval,
respectively, by $a_1$ and $a_2$.  We then introduce a contour $\mathcal{C}_1$ as the union of four segments
$\gamma_j,j=1,2,3,4$. Here
$$\gamma_1=u-iv_0h,u\in [a_l,a_r],\ \gamma_2=u+iv_0h,u\in [a_l,a_r],
$$$$\gamma_3=a_l+iv,v\in
[-v_0h,v_0h],\ \gamma_4=a_r+iv,v\in [-v_0h,v_0h],$$ where $a_l$ is
any positive value smaller than the left end point of (\ref{a2}),
$a_r$ any value larger than the right end point of (\ref{a2}), and
$v_0$ is specified in (\ref{f63}).
Assume that on the contour $\mathcal{C}_1$
\begin{equation}
\label{d1}
\Big|1-c_n(\underline{m}_n^0(z))^2\int\frac{t^2dH_n(t)}{(1+t\underline{m}_n^0(z))^2}\Big|\geq
M_1\sqrt{v},
\end{equation}
where $\Im(z)=v>0$, $M_1$ is a positive constant and $\underline{m}_n^0(z)$ is the
Stieltjes transform of the distribution function
$\underline{F}^{c_n,H_n}(x)$  which is
obtained from $\underline{F}^{c,H}(x)$ with $c$ and $H$ replaced by
$c_n$ and $H_n$. Also, on the contour $\mathcal{C}_1$ we assume that
\begin{equation}\label{f11}
\int\frac{dH_n(t)}{|1+tE\underline{m}_n(z)|^{4}}<M
\end{equation}
and that
 \begin{equation}\label{g38}
\int\frac{dH_n(t)}{|1+t\underline{m}_n^0(z)|^{4}}<M,
\end{equation}
where $E\underline{m}_n(z)$ is the expectation of the Stieltjes
transform of $F^{\bbB_n}$ and $M$ is a constant independent of $n$ and $z$. The main results are stated below.

\begin{theorem}\label{rem2}
Suppose that

\begin{itemize}
\item[1)]  $h = h(n)$ is a sequence of positive constants satisfying
\begin{equation}\label{band}
\lim_{n\to \infty}nh^{5/2}=\infty, \lim_{n\to \infty}nh^3=0, \ \ \lim_{n\to \infty}h=0 \ ;
\end{equation}

\item[2)] $K(x)$ satisfies (\ref{a25})-(\ref{f63}) and is analytic on
open interval including
$$
[\frac{a_2-a_1}{h},\frac{a_1-a_2}{h}] \ ;
$$

\item[3)]  $X_{ij}$ are i.i.d. with $EX_{11}=0$, $Var(X_{11})=1$,
$EX_{11}^4=3$ and
$EX_{11}^{16}<\infty$,  $c_n\rightarrow c\in (0,1)$ ;

\item[4)] $\bbT$ is a $p\times p$ non-random symmetric positive
definite matrix with spectral norm bounded above by a positive
constant such that $H_n=F^{\bbT}$ converges weakly to a
distribution $H$. Also, $H_n$ satisfies conditions (\ref{d1})-(\ref{g38});

\item[5)] $F^{c,H}(x)$ has a compact support $[a,b]$ with $a>0$;

\item[6)] the function $K(x)$ and $h$ satisfy \begin{equation}\label{g52}
nh^2\Big[\int^{+\infty}_{\frac{x-a}{h}}yK(y)dy+\int^{\frac{x-b}{h}}_{-\infty}
yK(y)dy\Big]\rightarrow 0,\quad \int y^2K(y)f_{c_n,H_n}''(y_0)dy<\infty
\end{equation}
and \begin{equation}\label{g51}
 \quad nh\Big[1-\int^{\frac{x-a}{h}}_{\frac{x-b}{h}} K(y)dy\Big]\rightarrow
0,
\end{equation}
where $f_{c_n,H_n}(x)$ is the density function of $F^{c_n,H_n}(x)$ and
$y_0=t(x-yh)+(1-t)x$ with $t\in (0,1)$ and $x\in (a,b)$.
\end{itemize}

Then, as $n\rightarrow\infty$, the limiting finite dimensional distributions
of the processes of
\begin{equation}\label{g49}
 nh\Big(f_n(x)-f_{c_n,H_n}(x)\Big)
, \quad x\in (a,b)
\end{equation}
are  multivariate normal with mean zero and covariance matrix $\sigma^2 \, I$,
where
 \ignore{
$$
Cov(W(x),W(y))=
\begin{cases}
\sigma^2,&\text{ if x=y}\\
0&  \text{otherwise},
\end{cases}
$$
where
} $$
\sigma^2=-\frac{1}{2\pi^2}\int_{-\infty}^{+\infty}
\int_{-\infty}^{+\infty}K'(u_1) K'(u_2)\ln (u_1-u_2)^2du_1du_2.
$$
\end{theorem}

\begin{remark}\label{rem1}
When $\bbT$ is the identity matrix, (\ref{d1}) is true, which will be verified
in Appendix 2, and conditions (\ref{f11}) and
(\ref{g38}) also hold, see, (6.30) in \cite{ker} and (\ref{f70}). For general
$\bbT$, (\ref{f11}) may be removed at the cost of higher moment
of $H_n(t)$ and of a more stringent bandwidth. See Lemma \ref{lem3}
in Appendix 2.
\end{remark}
\begin{remark}
It is easy to check that the Gaussian kernel function satisfies conditions specified in Theorem \ref{rem2}.
\end{remark}

 Theorem \ref{rem2} is actually a corollary of the following theorem.

\begin{theorem} \label{theo1}
 When the conditions (\ref{g52}), (\ref{g51}) and $\lim_{n\to \infty}nh^3=0$ in Theorem \ref{rem2}
  are removed with the remaining conditions unchanged, Theorem \ref{rem2} holds as well if the
processes (\ref{g49}) are replaced by the processes
$$nh\Big[f_n(x)-\frac{1}{h}\int^b_a
K(\frac{x-y}{h})dF^{c_n,H_n}(y)\Big].
$$
\end{theorem}

We evaluate the quality of the estimate $f_n(x)$ by the mean
integrated square error
\begin{eqnarray*}
L&=&E\Big(\int_a^b(f_n(x)-f_{c_n,H_n}(x))^2dx\Big)\\
&=&\int_a^b\Big(Bias(f_n(x))\Big)^2dx+\int_a^b Var(f_n(x))dx,
\end{eqnarray*}
where $Bias(f_n(x))=Ef_n(x)-f_{c_n,H_n}(x)$. It is easy to verify that
(see \cite{bw} and \cite{p1})
$$ \frac{1}{h}\int K(\frac{x-y}{h})dF^{c_n,H_n}(y)-f_{c_n,H_n}(x)=\frac{1}{2}h^2(f^{c,H}(x))''\int x^2K(x)dx+O(h^3).$$
Although it is not rigorous from Theorem \ref{theo1} we roughly have
$$
Ef_n(x)-\frac{1}{h}\int
K(\frac{x-y}{h})dF^{c_n,H_n}(y)=o(\frac{1}{nh})
$$
and
$$
Var(f_n(x))=\frac{\sigma^2}{n^2h^2}+o(\frac{\sigma^2}{n^2h^2}).
$$
These gives
$$
L=\Big(\frac{1}{2}h^2(f_{c_n,H_n}(x))''\int
x^2K(x)dx+O(h^3)+o(\frac{1}{nh})\Big)^2+\frac{\sigma^2(b-a)}{n^2h^2}+o(\frac{\sigma^2}{n^2h^2}).
$$
Differentiating the above with respect to $h$ and setting it equal
to zero, we see that the asymptotic optimal bandwidth is
\begin{equation}
h_{*}=\Big(\frac{\sigma^2(b-a)}{2n^2c_1^2}\Big)^{1/6},
\end{equation}
where $c_1=\frac{1}{2}(f_{c_n,H_n}(x))''\int x^2K(x)dx<\infty$. This is different from the asymptotic optimal bandwidth $O(1/n^{1/5})$ in classical density estimates (see \cite{bw}).

As for $F_n(x)=\int_0^x\, f_n(y)\,dy$, we have the following result.

\begin{theorem}\label{theo2} In addition to assumptions 2), 3),
4) and 5) in Theorem \ref{rem2}, suppose that
$$\lim\limits_{n\rightarrow\infty}nh^3\sqrt{\ln\frac{1}{h}}\rightarrow \infty,\quad
\lim\limits_{n\rightarrow\infty}h\rightarrow0.$$ Then, as
$n\rightarrow\infty$, the limiting finite dimensional distributions of the
processes of
\begin{equation}\label{g50}
\frac{n}{\sqrt{\ln\frac{1}{h}}}\Big(F_n(x)-\int^x_{-\infty}\Big[\frac{1}{h}\int
K(\frac{t-y}{h})dF^{c_n,H_n}(y)\Big]dt\Big)
\end{equation}
are multivariate normal with mean zero and
covariance matrix $ { 1 \over 2\pi^2} I$.
\ignore{
converge to those of a Gaussian process $W_1(x)$ with mean zero and
covariance
$$
Cov(W_1(x),W_1(y))=
\begin{cases}
\frac{1}{\pi^2},&\text{ if x=y}\\
0&  \text{otherwise}.
\end{cases}
$$
}
\end{theorem}

\begin{remark}
We conjecture that Theorem \ref{theo2} is still true if
we substitute $F^{c_n,H_n}(x)$ for $\int^x_{-\infty}\big[\frac{1}{h}\int
K(\frac{t-y}{h})dF^{c_n,H_n}(y)\big]dt$.
The convergence rate $n/\sqrt{\ln\frac{1}{h}}$ is consistent with the
conjectured convergence rate $n/\sqrt{\log n}$ of the empirical spectral
 distributions of sample covariance matrices to the MP type distribution.
\end{remark}

The paper is organized as follows.
Theorem \ref{theo1} is proved in Section 2 and Section 3.
In Section 4 we present the proof of Theorem \ref{theo2}.
Some technical lemmas are given in Appendix 1.  Appendix 2 deals
with Remark \ref{rem1} and Theorem \ref{rem2}
and Appendix 3 gives the derivation of the variances and means in Theorem \ref{theo1} and Theorem \ref{theo2}.

\section{Finite dimensional convergence of the processes}

Throughout the paper, to save notation, $M$ may stand for different
constants on different occasions. This and the subsequent sections deal with Theorem \ref{theo1} and the argument for handling $nh\Big(\frac{1}{h}\int^b_aK(\frac{x-y}{h})dF^{c_n,H_n}(y)-f_{c_n,H_n}(x)\Big)$ is given at the end of Appendix 2.

 Following the truncation steps in \cite{b2} we may
truncate and re-normalize the random variables as follows
\begin{equation}\label{f46}
|X_{ij}|\leq \tau_nn^{1/2},\  EX_{ij=0},\  EX_{ij}^2=1,
\end{equation}
where $\tau_nn^{1/3}\rightarrow\infty$. Based on this one may then
verify that
\begin{equation}\label{f49}
EX_{11}^4=3+O(\frac{1}{n}).
\end{equation}

For any finite constants $l_1,\cdots,l_r$ and $x_1,\cdots,x_r\in [a,b]$, by Cauchy's formula
\begin{eqnarray}
&&nh\Big(\sum\limits_{j=1}^rl_j\big(f_n(x_j)-\frac{1}{h}\int
K(\frac{x_j-y}{h})dF^{c_n,H_n}(y)\big)\Big)\non &=&-\frac{1}{2\pi
i}\oint_{\mathcal{C}_1}
\sum\limits_{j=1}^rl_jK(\frac{x_j-z}{h})X_n(z)dz,\label{a3*}
\end{eqnarray}
where $X_n(z)=tr(\bbA_n-z\bbI)^{-1}-nm_{F^{^{c_n,H_n}}}(z)$ and $\mathcal{C}_1$ is defined in the introduction.

 From Fubini's theorem and (\ref{f63}) we obtain for $j=0,1,2$.
$$
\int^{a_r}_{a_l}\Big[\frac{1}{h}\int^{v_0}_{0}|K^{(j)}(\frac{x-u}{h}+iv)|dv\Big]du=\int^{v_0}_{0}\Big[\frac{1}{h}\int^{a_r}_{a_l}|K^{(j)}(\frac{x-u}{h}+iv)|du\Big]dv<\infty.
$$
This implies for $u\in [a_l,a_r]$
\begin{equation}
\label{g3}\frac{1}{h}\int^{v_0}_{0}|K^{(j)}(\frac{x-u}{h}+iv)|dv<\infty,\quad j=0,1,2.\end{equation}

For the sake of simplicity, write $\bbA=\bbA_n$. We now introduce some notation.
 Define $\bbA(z)=\bbA-z\bbI$,
$\bbA_k(z)=\bbA(z)-\bbs_k\bbs_k^T$, and $\bbs_k=\bbT^{1/2}\bbx_k$
with $\bbx_k$ being the $k$th column of $\bbX_n$. Let
$E_k=E(\cdot|\bbs_1,\cdots,\bbs_k)$ and $E_0$ denote the
expectation. Set
$$
\beta_k(z)=\frac{1}{1+\bbs_k^T\bbA_k^{-1}(z)\bbs_k},\
\eta_k(z)=
\bbs_k^T\bbA_k^{-1}(z)\bbs_k-\frac{1}{n}tr\bbT\bbA_k^{-1}(z),
$$
$$
b_1(z)=\frac{1}{1+Etr\bbT\bbA_1^{-1}(z)/n},\
\beta_k^{tr}(z)=\frac{1}{1+tr\bbT\bbA_k^{-1}(z)/n}.
$$
We frequently use the following equalities:
\begin{equation}
\label{b23}\bbA^{-1}(z)-\bbA_k^{-1}(z)=-\beta_k(z)\bbA_k^{-1}(z)\bbs_k\bbs_k^T\bbA_k^{-1}(z);
\end{equation}
\begin{equation}
\label{b22}\beta_1=b_1-b_1\beta_1\xi_1(z)=b_1-b_1^2\xi_1(z)+b_1^2\beta_1\xi_1^2(z)
\end{equation}
where
$\xi_1(z)=\bbs_1^T\bbA_1^{-1}(z)\bbs_1-En^{-1}tr\bbA_1^{-1}(z)\bbT$.
At this moment, we would point out that the length of the vertical
lines of the contour of integral in (\ref{a3*}) converges to zero. As
a consequence, except $|b_1(z)|$ we can not expect $|\beta_k(z)|$
and $|\beta_k^{tr}(z)|$ to be bounded above by constants although
they are bounded by $|z|/|v|$ (see \cite{b4}) (of course $v\neq 0$ in the cases of interest). Instead, the moments of
$\beta_k(z)$ and $\beta_k^{tr}(z)$ are proved to be bounded. We
summarize such estimates in Lemma \ref{lem1} in Appendix 1. Sometimes we deal with
the term $\frac{1}{n}tr\bbT\bbA_1^{-1}(z)\bbT\bbA_1^{-1}(\bar z)$ in
the following way: One may verify that
$$
Im(1+\frac{1}{n}tr\bbT\bbA_k^{-1}(z))\geq v\lambda_{\min}(\bbT)
\frac{1}{n}tr\bbA_k^{-1}(z)\bbA_k^{-1}(\bar z),
$$
which implies that
\begin{equation}\label{f44}
|\beta^{tr}_k(z)\frac{1}{n}tr\bbT\bbA_k^{-1}(z)\bbT\bbA_k^{-1}(\bar
z)|\leq\frac{M}{|v|}.
\end{equation}
We also frequently use the fact that $\|\bbA_k^{-1}(z)\|\leq 1/|v|$.

Write
$$
tr\bbA^{-1}(z)-Etr\bbA^{-1}(z)=\sum\limits_{k=1}^n\Big(E_ktr \bbA^{-1}(z)-E_{k-1}tr \bbA^{-1}(z)\Big)
$$
$$=\sum\limits_{k=1}^n\Big(E_k-E_{k-1}\Big)tr\Big[\bbA^{-1}(z)-\bbA_k^{-1}(z)\Big]
=-\sum\limits_{k=1}^n\Big(E_k-E_{k-1}\Big)\Big[\beta_k(z)\bbs_k^T\bbA_k^{-2}(z)\bbs_k\Big]
$$
\begin{equation}
=-\sum\limits_{k=1}^n\Big(E_k-E_{k-1}\Big)\Big[\log\beta_k(z)\Big]',\label{g39}
\end{equation}
where in the third step one uses (\ref{b23}) and the derivative in the last equality is with respect to $z$.
We then obtain from integration by parts that
\begin{equation}\label{b12}
\frac{1}{2\pi i}\oint
K(\frac{x-z}{h})(tr\bbA^{-1}(z)-Etr\bbA^{-1}(z))dz
\end{equation}
$$
=-\frac{1}{2\pi i}\sum\limits_{k=1}^n(E_k-E_{k-1})\oint
K(\frac{x-z}{h})\Big[\log\beta_k(z)\Big]'dz
$$
\begin{equation}\label{h6}
=\frac{1}{h}\frac{1}{2\pi i}\sum\limits_{k=1}^n(E_k-E_{k-1})\oint
K'(\frac{x-z}{h})\log \Big(\frac{\beta^{tr}_k(z)}{\beta_k(z)}\Big)dz
\end{equation}
$$
=\frac{1}{h}\frac{1}{2\pi i}\sum\limits_{k=1}^n(E_k-E_{k-1})\oint
K'(\frac{x-z}{h})\log \Big(1+\beta^{tr}_k(z)\eta_k(z)\Big)dz
$$
\begin{equation}
=\frac{1}{h}\frac{1}{2\pi i}\sum\limits_{k=1}^n(E_k-E_{k-1})\oint
K'(\frac{x-z}{h})\Big(\beta^{tr}_k(z)\eta_k(z)+e_k(z)\Big)dz\label{f41}
\end{equation}
where
$$
e_k(z)=\log (1+\beta^{tr}_k(z)\eta_k(z))-\beta^{tr}_k(z)\eta_k(z).
$$

Below, consider $z\in\gamma_2$, the top horizontal line of the contour, unless it is further specified. We remind readers that $v=v_0h$ on $\gamma_2$. The next aim is to prove that
\begin{equation}\label{f2}
\frac{1}{h}\sum\limits_{k=1}^n(E_k-E_{k-1})\int
K'(\frac{x-z}{h})e_k(z)dz\stackrel{i.p.}\longrightarrow0.
\end{equation}
By Lemma \ref{lem8}, we have for $m=2,4,6$
\begin{equation}\label{g1}
E\Big(|\eta_k(z)|^m|\bbA_k^{-1}(z)\Big)\leq \frac{M}{n^{m/2}}
\Big[\frac{1}{n}tr\bbT\bbA_k^{-1}(z)\bbT\bbA_k^{-1}(\bar
z)\Big]^{m/2}.
\end{equation}
This,  together with Lemma \ref{lem1} in Appendix 1 and (\ref{f44}), gives
\begin{equation}\label{h1}
E|\beta^{tr}_k(z)\eta_k(z)|^4=E\Big(|\beta^{tr}_k(z)|^4E(|\eta_k(z)|^4|\bbA_k^{-1}(z))\Big)\leq\frac{M}{n^2v^2}.
\end{equation}
It follows that
$$
\sum\limits_{k=1}^nP(|\beta^{tr}_k(z)\eta_k(z)|\geq 1/2)\leq
2^4\sum\limits_{k=1}^nE|\beta^{tr}_k(z)\eta_k(z)|^4\leq
\frac{M}{nv^2}.
$$
 Via (\ref{f63}), (\ref{h1}) and the inequality
  \begin{equation}|\label{h9}
  \log (1+x)-x|\leq
M|x|^2, \ \mbox{for}\ |x|\leq 1/2,
\end{equation}  we obtain
$$
\frac{1}{h^2}E\Big|\sum\limits_{k=1}^n(E_k-E_{k-1})\int
K'(\frac{x-z}{h})e_k(z)I(|\beta^{tr}_k(z)\eta_k(z)|< 1/2)du\Big|^2
$$
\begin{equation}\label{g2}
\leq \frac{M}{h^2}\sum\limits_{k=1}^nE\Big|\int
K'(\frac{x-z}{h})e_k(z)I(|\beta^{tr}_k(z)\eta_k(z)|< 1/2)du\Big|^2
\end{equation}
$$
\leq \frac{M}{h^2}\sum\limits_{k=1}^n\Big[\int\int
|K'(\frac{x-z_1}{h})K'(\frac{x-z_2}{h})|\Big(E|(\beta^{tr}_k(z_1)\eta_k(z_1))|^4
$$$$\times E|(\beta^{tr}_k(z_2)\eta_k(z_2))|^4\Big)^{1/2}du_1du_2\Big]
\leq\frac{M}{nv^2}.
$$
Thus, (\ref{f2}) is proven. Therefore on $\gamma_2$
\begin{equation}\label{g41}
(\ref{b12})=\frac{1}{2\pi i}\sum\limits_{k=1}^nY_k(x)+o_p(1),
\end{equation}
where
$$
Y_k(x)=E_k\Big[\frac{1}{h}\int
K'(\frac{x-z}{h})\Big(\beta^{tr}_k(z)\eta_k(z)\Big)dz\Big].
$$

Apparently, $Y_k(z)$ is a martingale difference so that we may
resort to the CLT for martingale (see Theorem 35.12 in \cite{bili}). As in (\ref{g2}), by (\ref{f63}) and (\ref{h1})
we have
$$
\sum\limits_{k=1}^nE|Y_k(z)|^4\leq\frac{M}{nv^2}.
$$
which ensures the Lyapunov condition for the CLT is
satisfied.

Thus, it is sufficient to investigate the limit of the following
covariance function
\begin{eqnarray}
&&-\frac{1}{4\pi^2}\sum\limits_{k=1}^nE_{k-1}[Y_k(x_1)Y_k(x_2)] \non
&=&-\frac{1}{4h^2\pi^2}\int\int K'(\frac{x_1-z_1}{h})
K'(\frac{x_2-z_2}{h})\mathcal{C}_{n1}(z_1,z_2)dz_1dz_2,\label{f48}
\end{eqnarray}
where
$$
\mathcal{C}_{n1}(z_1,z_2)=\sum\limits_{k=1}^nE_{k-1}\Big[E_k\Big(\beta^{tr}_k(z_1)\eta_k(z_1)\Big)E_k\Big(\beta^{tr}_k(z_2)\eta_k(z_2)\Big)\Big].
$$


By (\ref{f44}), (\ref{g1}) and (\ref{b17})
\begin{eqnarray}
&&E\Big(|(\beta_k^{tr}(z)-b_1(z))\eta_k(z)|^2\Big|\bbA_k^{-1}(z)\Big)\non
&\leq&
\frac{M}{n^3}E\Big(|\beta_k^{tr}(z)b_1(z)(tr\bbA^{-1}(z)\bbT-Etr\bbA^{-1}(z)\bbT)|^2\non
&&\quad \times \frac{1}{n}tr\bbT\bbA_k^{-1}(z)\bbT\bbA_k^{-1}(\bar
z)\Big|\bbA_k^{-1}(z)\Big)\non &\leq
&\frac{M}{n^3v}|\beta_k^{tr}(z)||tr\bbA^{-1}(z)\bbT-Etr\bbA^{-1}(z)\bbT|^2.
\end{eqnarray}
This and Lemma \ref{lem1} lead to
$$
E\Big|(\beta_k^{tr}(z_1)-b_1(z_1))\eta_k(z_1)E_k\Big((\beta_k^{tr}(z_2)-b_1(z_2))\eta_k(z_2)\Big)\Big|
$$$$\leq \Big[E|(\beta_k^{tr}(z_1)-b_1(z_1))\eta_k(z_1)|^2E|(\beta_k^{tr}(z_2)-b_1(z_2))\eta_k(z_2)|^2\Big]^{1/2}
 \leq\frac{M}{n^3v^4}
$$
and
$$
E\Big|(\beta_k^{tr}(z_1)-b_1(z_1))\eta_k(z_1)E_k(\eta_k(z_2))\Big|
\leq\frac{M}{n^2v^{5/2}}.
$$
It follows that
\begin{equation}\label{g5}
E\Big|\mathcal{C}_n(z_1,z_2)-b_1(z_1)b_1(z_2)\sum\limits_{k=1}^nE_{k-1}\Big(E_k\eta_k(z_1)E_k\eta_k(z_2)\Big)\Big|\leq\frac{M}{nv^{5/2}}.
\end{equation}

Note that for any non-random matrices $\bbB$ and $\bbC$
\begin{eqnarray} \label{i1}
&&E(\bbs_1^T\bbC \bbs_1-tr\bbC)(\bbs_1^T\bbB\bbs_1-tr\bbB) \non & =&
n^{-2}(EX_{11}^4-|EX_{11}^2|^2-2)\sum\limits_{i=1}^p(\bbT^{1/2}\bbC\bbT^{1/2})_{ii}(\bbT^{1/2}\bbB\bbT^{1/2})_{ii}\non
&&+|EX_{11}^2|^2n^{-2}tr\bbT^{1/2}\bbC\bbT\bbB^T\bbT^{1/2}+n^{-2}tr\bbT^{1/2}\bbC\bbT\bbB\bbT^{1/2}.
\end{eqnarray}
This implies that
\begin{eqnarray}
&&b_1(z_1)b_1(z_2)\sum\limits_{k=1}^nE_{k-1}\Big(E_k\eta_k(z_1)E_k\eta_k(z_2)\Big)
\non
&=&(EX_{11}^4-3)b_1(z_1)b_1(z_2)\mathcal{C}_{n1}(z_1,z_2)+2b_1(z_1)b_1(z_2)\mathcal{C}_{n2}(z_1,z_2)\label{g7}\\
&=&2b_1(z_1)b_1(z_2)\mathcal{C}_{n2}(z_1,z_2)+O(\frac{1}{nv^2}),\label{g6}
\end{eqnarray}
where
$$
\mathcal{C}_{n1}(z_1,z_2)=\frac{1}{n^2}\sum\limits_{k=1}^n\sum\limits_{j=1}^p(E_k(\bbT^{1/2}\bbA_k^{-1}(z_1)\bbT^{1/2})_{jj}E_k(\bbT^{1/2}\bbA_k^{-1}(z_2)\bbT^{1/2})_{jj},
$$
$$
\mathcal{C}_{n2}(z_1,z_2)=\frac{1}{n^2}\sum\limits_{k=1}^ntr\bbT^{1/2}E_k(\bbA_k^{-1}(z_1))\bbT
E_k(\bbA_k^{-1}(z_2))\bbT^{1/2},
$$
and in the last step one uses (\ref{f49}) and the fact that
$|(E_k(\bbT^{1/2}\bbA_k^{-1}(z_1)\bbT^{1/2})_{jj}|\leq\frac{1}{v}$.

The next aim is to convert the random variables involved in
$\mathcal{C}_{n2}(z_1,z_2)$ to their corresponding expectations. To
this end, we introduce more notation and estimates, and establish a
lemma. Define
$$\bbA_{kj}(z)=\bbA(z)-\bbs_k\bbs_k^T-\bbs_j\bbs_j^T,
\beta_{kj}(z)=\frac{1}{1+\bbs_j^T\bbA^{-1}_{kj}(z)\bbs_j},
$$
 $$
b_{12}(z)=\frac{1}{1+n^{-1}Etr\bbT\bbA^{-1}_{12}(z)}, \quad
\beta_{kj}^{tr}(z)=\frac{1}{1+n^{-1}tr\bbA^{-1}_{kj}(z)}
$$
and
$$
\xi_{kj}(z)=\bbs_j^T\bbA^{-1}_{kj}(z)\bbs_j-En^{-1}tr\bbA^{-1}_{kj}(z)\bbT,\quad
\eta_{kj}(z)=\bbs_j^T\bbA^{-1}_{kj}(z)\bbs_j-n^{-1}tr\bbA^{-1}_{kj}(z)\bbT
.
$$
Note that
\begin{equation}\label{f6}
\bbA_{k}^{-1}(z)-\bbA_{kj}^{-1}(z)=-\beta_{kj}(z)\bbA_{kj}^{-1}(z)\bbs_j\bbs_j^T\bbA_{kj}^{-1}(z)
\end{equation}
and (see Lemma 2.10 of \cite{b4}) for any $p\times p$ matrix $\bbD$
\begin{equation}\label{f10}
|tr(\bbA_k^{-1}(z)-\bbA_{kj}^{-1}(z))\bbD|\leq\frac{\|\bbD\|}{v}.
\end{equation}
By Lemma \ref{lem1} in Appendix 1 and (\ref{f10}) we have
\begin{equation}\label{g9}
\frac{1}{n^4}E|tr\bbA_{kj}^{-1}(z)\bbT-Etr\bbA_{kj}^{-1}(z)\bbT|^8\leq\frac{M}{n^8v^{12}\|\bbT\|^8},\quad
E|\xi_{kj}(z)|^8\leq\frac{M}{n^4v^4}.
\end{equation}
By Lemma \ref{lem1} we have
\begin{equation}\label{f14}
E\frac{1}{n}tr\bbA^{-1}(z)\bbA^{-1}(\bar
z)=\frac{1}{v}Im(E\frac{1}{n}tr\bbA^{-1}(z))\leq\frac{M}{v},
\end{equation}
which, together with (\ref{f10}), implies that
\begin{equation}\label{f15}
E\frac{1}{n}tr\bbA_{kj}^{-1}(z)\bbA_{kj}^{-1}(\bar
z)\leq\frac{M}{v}.
\end{equation}
By (\ref{f10}) and the fact that $b_1(z)$ is bounded, given in Lemma
\ref{lem1}, it is straightforward to verify that
$|b_1(z)-b_{12}(z)|\leq \frac{1}{nv^2}$ and hence
\begin{equation}\label{f3*}
|b_{12}(z)|\leq M.
\end{equation}

We are now in a position to state Lemma \ref{lem7}:
\begin{lemma}\label{lem7}
For non-random matrix $\bbD$
\begin{eqnarray}
&&E\Big|\frac{1}{n}tr\bbD\bbA_k^{-1}(z_1)\bbT
E_k(\bbA_k^{-1}(z_2))-E\Big(\frac{1}{n}tr\bbD\bbA_k^{-1}(z_1)\bbT
E_k(\bbA_k^{-1}(z_2))\Big)\Big|^2\non &\ & \ \leq
\begin{cases}
 \frac{M}{n^2v^5} &\text{when $\|\bbD\|\leq M$} \\
           \frac{M}{n^2v^6} &\text{when $\frac{1}{n}tr\bbD\bbD^*\leq
           M$}.
\end{cases}\label{f5}
\end{eqnarray}
\end{lemma}
\begin{remark}
Checking on the argument of Lemma \ref{lem7}, we see that Lemma
\ref{lem7} holds as well when we replace $E_k(\bbA_k^{-1}(z_2))$ by
$\bbA_k^{-1}(z_2)$. The main difference of arguments is that we do
not need to distinguish between the cases $j<k$ and $j>k$ when
dealing with the latter.
\end{remark}

\begin{proof} We begin with a martingale decomposition of random variable of
interest:
$$
\frac{1}{n}tr\bbD\bbA_k^{-1}(z_1)\bbT
E_k(\bbA_k^{-1}(z_2))-E\Big(\frac{1}{n}tr\bbD\bbA_k^{-1}(z_1)\bbT
E_k(\bbA_k^{-1}(z_2))\Big)
$$
$$
=\frac{1}{n}\sum\limits_{j\neq
k}^{n}(E_j-E_{j-1})\Big[tr\bbD\bbA_k^{-1}(z_1)\bbT
E_k(\bbA_k^{-1}(z_2))\Big]
$$
$$
=\frac{1}{n}\sum\limits_{j\neq
k}^{n}(E_j-E_{j-1})\Big[tr\bbD\bbA_k^{-1}(z_1)\bbT
E_k(\bbA_{k}^{-1}(z_2))-tr\bbD\bbA_{kj}^{-1}(z_1)\bbT
E_k(\bbA_{kj}^{-1}(z_2))\Big]
$$
$$
=\frac{1}{n}\sum\limits_{j\neq
k}^{n}(E_j-E_{j-1})(\delta_1+\delta_2+\delta_3),
$$
where, via (\ref{f6}),
$$
\delta_1=\beta_{kj}(z_1)\bbs_j^T\bbA_{kj}^{-1}(z_1)\bbT
E_k\Big(\beta_{kj}(z_2)\bbA_{kj}^{-1}(z_2)\bbs_j\bbs_j^T\bbA_{kj}^{-1}(z_2)\Big)\bbD\bbA_{kj}^{-1}(z_1)\bbs_j
$$
$$
\delta_2=-\beta_{kj}(z_1)\bbs_j^T\bbA_{kj}^{-1}(z_1)\bbT
E_k\Big(\bbA_{kj}^{-1}(z_2)\Big)\bbD\bbA_{kj}^{-1}(z_1)\bbs_j
$$
and
$$
\delta_3=-tr\bbD\bbA_{kj}^{-1}(z_1)\bbT
E_k\Big(\beta_{kj}(z_2)\bbA_{kj}^{-1}(z_2)\bbs_j\bbs_j^T\bbA_{kj}^{-1}(z_2)\Big).
$$

Note that
\begin{equation}\label{f7}
|\beta_{kj}|\|\bbs_j^T\bbA_{kj}^{-1}(z)\|^2=|\beta_{kj}\bbs_j^T\bbA_{kj}^{-1}(z)\bbA_{kj}^{-1}(\bar
z)\bbs_j|\leq \frac{1}{v}
\end{equation}
which implies that
\begin{equation}\label{f12}
|\delta_1|\leq\frac{1}{v^2\|\bbD\|}.
\end{equation}

Write
\begin{equation}\label{f13}
\beta_{kj}(z)=b_{12}(z)-\beta_{kj}(z)b_{12}(z)\xi_{kj}(z)=b_{12}(z)-b_{12}^2(z)\xi_{kj}(z)+\beta_{kj}(z)b_{12}^2(z)\xi^2_{kj}(z).
\end{equation}
This implies that when $j>k$,
$$
(E_j-E_{j-1})\delta_1=(E_j-E_{j-1})b_{12}(z_1)(\delta_{11}-\delta_{12}),
$$
where $\delta_{12}=\xi_{kj}(z_1)\delta_1$ and
$$
\delta_{11}=\bbs_j^T\bbA_{kj}^{-1}(z_1)\bbT
E_k\Big(\beta_{kj}(z_2)\bbG_k(z_2)\Big)\bbD\bbA_{kj}^{-1}(z_1)\bbs_j
$$$$\qquad \qquad\qquad-n^{-1}tr\bbT\bbA_{kj}^{-1}(z_1)\bbT
E_k\Big(\beta_{kj}(z_2)\bbG_k(z_2)\Big)\bbD\bbA_{kj}^{-1}(z_1)
$$
 with
$\bbG_k(z_2)=\bbA_{kj}^{-1}(z_2)\bbs_j\bbs_j^T\bbA_{kj}^{-1}(z_2)$.
When $\|\bbD\|\leq M$ we conclude from (\ref{f7}), (\ref{f12}), (\ref{f15}) and Lemma \ref{lem8} that
$$
E|\frac{1}{n}\sum\limits_{j\neq
k}^{n}(E_j-E_{j-1})(\delta_{11}+\delta_{12})|^2\leq
\frac{1}{n^2}\sum\limits_{j\neq
k}^{n}E|\delta_{11}|^2+E|\delta_{12}|^2\leq\frac{M}{n^2v^5\|\bbD\|^2}.
$$
When $\frac{1}{n}tr\bbD\bbD^*\leq M$, by Lemma \ref{lem8} and
(\ref{f7})
$$
E|\delta_{11}|^8\leq \frac{M}{n^4v^{24}}(\frac{1}{n}tr\bbD\bbD^*)^4\leq \frac{M}{n^4v^{24}}.
$$
This, together with (\ref{g9}) and Lemma \ref{lem1}, implies
$$
E|\delta_{12}|^2\leq ME|\xi_{kj}\beta_{kj}\delta_{11}|^2+\frac{M}{n^2}E|\xi_{kj}\beta_{kj}tr\bbT\bbA_{kj}^{-1}(z_1)\bbT
E_k\Big(\beta_{kj}(z_2)\bbG_k(z_2)\Big)\bbD\bbA_{kj}^{-1}(z_1)|^2
$$
$$
\leq \frac{M}{nv^6},
$$
because via (\ref{f7}) and Holder's inequality
\begin{eqnarray}
&&\frac{1}{n^4}E|tr\bbT\bbA_{kj}^{-1}(z_1)\bbT
E_k\Big(\beta_{kj}(z_2)\bbG_k(z_2)\Big)\bbD\bbA_{kj}^{-1}(z_1)|^8\non
&\leq&\frac{M}{v^{16}}E(\frac{1}{n}tr\bbA_{kj}^{-1}(z_1)\bbA_{kj}^{-1}(\bar z_1))^4(\frac{1}{n}tr\bbD\bbD^*)^4\non
&\leq& \frac{M}{v^{20}}E|\frac{1}{n}tr\bbA_{kj}^{-1}(z_1)-E\frac{1}{n}tr\bbA_{kj}^{-1}(z_1)|^4+\frac{M}{v^{20}}|E\frac{1}{n}tr\bbA_{kj}^{-1}(z_1)|^4\leq
\frac{M}{v^{20}}\label{g18}.
\end{eqnarray}
These give
$$
E|\frac{1}{n}\sum\limits_{j\neq
k}^{n}(E_j-E_{j-1})(\delta_{11}+\delta_{12})|^2\leq
\frac{1}{n^2}\sum\limits_{j\neq
k}^{n}E|\delta_{11}|^2+E|\delta_{12}|^2\leq\frac{M}{n^2v^6}.
$$

For handling the case $j<k$, we define
$\underline{\bbA}_{kj}^{-1}(z),\underline{\beta}_{kj}(z)$ and
$\underline{\xi}_{kj}(z)$ using
$\bbs_1,\cdots,\bbs_{j-1},\underline{\bbs}_{j+1},\cdots,\underline{\bbs}_{k-1},\underline{\bbs}_{k+1},\cdots,\underline{\bbs}_n$
as $\bbA_{kj}^{-1}(z),\beta_{kj}(z)$ and $\xi_{kj}(z)$ are defined
using
$\bbs_1,\cdots,\bbs_{j-1},\bbs_{j+1},\cdots,\bbs_{k-1},\bbs_{k+1},\cdots,\bbs_n$.
Here $\underline{\bbs}_{1},\cdots,\underline{\bbs}_{n}$ are i.i.d.
copies of $\bbs_1$ and independent of $\{\bbs_j,j=1,\cdots,n\}$. Let
$$\alpha_{k1}=\bbs_j^T\bbA_{kj}^{-1}(z_1)\bbT
\underline{\bbA}_{kj}^{-1}(z_2)\bbs_j,\quad
\alpha_{k2}=\bbs_j^T\bbA_{kj}^{-1}(z_1)\bbD
\underline{\bbA}_{kj}^{-1}(z_2)\bbs_j.$$  Applying (\ref{f13}) and
the equality for $\underline{\beta}_{kj}(z_2)$ similar to
(\ref{f13}) yields
$$
(E_j-E_{j-1})\delta_1=(E_j-E_{j-1})\Big[\beta_{kj}(z_1)\underline{\beta}_{kj}(z_2)\alpha_{k1}
\alpha_{k2}\Big]
$$
$$
=b_{12}(z_1)b_{12}(z_2)[\delta_{13}+\delta_{14}+\delta_{15}+\delta_{16}+\delta_{17}+\delta_{18}],
$$
where
$$
\delta_{13}=(E_j-E_{j-1})\Big(\zeta_{kj1}\zeta_{kj2}\Big),
\delta_{14}=(E_j-E_{j-1})\Big(\zeta_{kj1}n^{-1}tr\bbA_{kj}^{-1}(z_1)\bbD
\underline{\bbA}_{kj}^{-1}(z_2)\bbT\Big),
$$
$$
\delta_{15}=(E_j-E_{j-1})\Big(\zeta_{kj2}n^{-1}tr\bbA_{kj}^{-1}(z_1)\bbT
\underline{\bbA}_{kj}^{-1}(z_2)\bbT\Big),
$$
$$
\delta_{16}=-(E_j-E_{j-1})\Big[\beta_{kj}(z_1)\xi_{kj}(z_1)\alpha_{k1}
\alpha_{k2}\Big],
$$
$$
\delta_{17}=-(E_j-E_{j-1})\Big[\underline{\beta}_{kj}(z_2)\underline{\xi}_{kj}(z_2)\alpha_{k1}
\alpha_{k2}\Big]
$$
and
$$
\delta_{18}=(E_j-E_{j-1})\Big[\beta_{kj}(z_1)\underline{\beta}_{kj}(z_2)\xi_{kj}(z_1)\underline{\xi}_{kj}(z_2)\alpha_{k1}
\alpha_{k2}\Big],
$$
with
$$
\zeta_{kj1}=\alpha_{k1}-n^{-1}tr\bbA_{kj}^{-1}(z_1)\bbT
\underline{\bbA}_{kj}^{-1}(z_2)\bbT,\quad
\zeta_{kj2}=\alpha_{k2}-n^{-1}tr\bbA_{kj}^{-1}(z_1)\bbD
\underline{\bbA}_{kj}^{-1}(z_2)\bbT.
$$

Consider $\|\bbD\|\leq M$ first. It follows from Lemma \ref{lem8},
(\ref{g9}) and (\ref{f15}) that (or see (\ref{h2}) in Appendix 1)
\begin{equation}\label{f16}
E|\zeta_{kj1}|^4\leq
\frac{M}{n^2v^4}E|\frac{1}{n}tr\bbA_{kj}^{-1}(z_1)\bbA_{kj}^{-1}(\bar
z_1)|^2\leq\frac{M}{n^2v^6},\ E|\zeta_{kj2}|^4\leq\frac{M}{n^2v^6\|\bbD\|^2}.
\end{equation}
Similarly, by (\ref{g9}) and (\ref{f15}), as in (\ref{g18}), we have
\begin{eqnarray}
&&E|n^{-1}tr\bbA_{kj}^{-1}(z_1)\bbD
\underline{\bbA}_{kj}^{-1}(z_2)\bbT|^4 \non &\leq&
\frac{M}{\|\bbD\|^4}\Big(E|n^{-1}tr\bbA_{kj}^{-1}(z_1)\bbA_{kj}^{-1}(\bar
z_1)|^4E|n^{-1}tr\underline{\bbA}_{kj}^{-1}(z_2)\underline{\bbA}_{kj}^{-1}(\bar
z_2)|^4\Big)^{1/2}\non
 &\leq&\frac{M}{\|\bbD\|^4v^4}.\label{g8}
\end{eqnarray}
In view of (\ref{f16}) and (\ref{f3*}),
$$
E|\frac{1}{n}\sum\limits_{j\neq
k}^{n}(E_j-E_{j-1})b_{12}(z_1)b_{12}(z_2)(\delta_{13})|^2\leq
\frac{M}{n^3v^6\|\bbD\|^2}.
$$
While (\ref{g8}) and (\ref{f3*}) yield
 $$
E|\frac{1}{n}\sum\limits_{j\neq
k}^{n}(E_j-E_{j-1})b_{12}(z_1)b_{12}(z_2)(\delta_{1j})|^2\leq
\frac{M}{n^2v^5\|\bbD\|^2},\ j=4,5.
$$
It follows from (\ref{f7}) that
$$
|\beta_{kj}(z_1)\alpha_{k1}\alpha_{k2}|\leq \frac{M}{v}\|\underline{\bbA}_{kj}^{-1}(z_2)\bbs_j\|^2=\frac{M}{v}\bbs_j^T\underline{\bbA}_{kj}^{-1}(\bar z_2)
\underline{\bbA}_{kj}^{-1}(z_2)\bbs_j.
$$
In the mean time we obtain from Lemma \ref{lem8}
\begin{equation}\label{g23}
E|\bbs_j^T\underline{\bbA}_{kj}^{-1}(\bar z_2)
\underline{\bbA}_{kj}^{-1}(z_2)\bbs_j-\frac{1}{n}tr\underline{\bbA}_{kj}^{-1}(\bar z_2)
\underline{\bbA}_{kj}^{-1}(z_2)\bbT|^8\leq\frac{M}{n^4v^{12}}.
\end{equation}
Thus we have
$$
E|\frac{1}{n}\sum\limits_{j\neq
k}^{n}(E_j-E_{j-1})b_{12}(z_1)b_{12}(z_2)(\delta_{16})|^2
\leq\frac{M}{n^2v^5}.
$$
Obviously, this estimate applies to the term involving
$\delta_{17}$. From (\ref{f7}) and Lemma \ref{lem1} we obtain
$$
E|\frac{1}{n}\sum\limits_{j\neq
k}^{n}(E_j-E_{j-1})b_{12}(z_1)b_{12}(z_2)(\delta_{18})|^2\leq\frac{M}{n^3v^6\|\bbD\|^2}.
$$
Summarizing the above we have
$$
E|\frac{1}{n}\sum\limits_{j\neq
k}^{n}(E_j-E_{j-1})b_{12}(z_1)b_{12}(z_2)(\delta_{1})|^2\leq\frac{M}{n^2v^5\|\bbD\|^2}.
$$

Consider $\frac{1}{n}tr\bbD\bbD^*\leq M$ next. By Lemma \ref{lem8}
\begin{equation}\label{g19}
E|\zeta_{kj2}|^8\leq\frac{M}{n^4}E(\frac{1}{n}tr\bbA_{kj}^{-1}(z_1)\bbD
\underline{\bbA}_{kj}^{-1}(z_2)\bbT\underline{\bbA}_{kj}^{-1}(\bar z_2)\bbD\bbA_{kj}^{-1}(\bar z_1)\bbT)^4\leq\frac{M}{n^4v^{16}}(\frac{1}{n}tr\bbD\bbD)^2.
\end{equation}
Observe that
\begin{equation}\label{g22}
E|n^{-1}tr\bbA_{kj}^{-1}(z_1)\bbD
\underline{\bbA}_{kj}^{-1}(z_2)\bbT|^8 \leq
\frac{M}{v^{12}}|n^{-1}tr\bbD\bbD^*|^4.
\end{equation}
This, together with (\ref{f16}) and (\ref{g8}), give
 $$
E|\frac{1}{n}\sum\limits_{j\neq
k}^{n}(E_j-E_{j-1})b_{12}(z_1)b_{12}(z_2)(\delta_{13})|^2\leq
\frac{M}{n^3v^7},
$$ and
$$
E|\frac{1}{n}\sum\limits_{j\neq
k}^{n}(E_j-E_{j-1})b_{12}(z_1)b_{12}(z_2)(\delta_{1j})|^2\leq
\frac{M}{n^2v^6},\ j=4,5.
$$
To deal with $\delta_{16}$, we obtain from (\ref{f7})
$$
E|\beta_{kj}(z_1)\xi_{kj}(z_1)\alpha_{k1}
\alpha_{k2}|^2\leq\frac{M}{v} E\Big(\sqrt{|\beta_{kj}(z_1)|}\xi_{kj}(z_1)\zeta_{kj2}\|\underline{\bbA}_{kj}^{-1}(z_2)\bbs_j\|\Big)^2
$$$$+\frac{M}{v}E\Big(\sqrt{|\beta_{kj}(z_1)|}\xi_{kj}(z_1)\|\underline{\bbA}_{kj}^{-1}(z_2)\bbs_j\|n^{-1}tr\bbA_{kj}^{-1}(z_1)\bbD
\underline{\bbA}_{kj}^{-1}(z_2)\bbT\Big)^2.
$$
We conclude from (\ref{g19}), (\ref{g9}), (\ref{g23}) and Lemma \ref{lem1}
$$E\Big(\sqrt{|\beta_{kj}(z_1)|}\xi_{kj}(z_1)\zeta_{kj2}\|\underline{\bbA}_{kj}^{-1}(z_2)\bbs_j\|\Big)^2
$$$$\leq (E|\xi_{kj}|^8E|\zeta_{kj2}|^8)^{1/4}(E|\beta_{kj}|^4E|\bbs_j^T\underline{\bbA}_{kj}^{-1}(\bar z_2)
\underline{\bbA}_{kj}^{-1}(z_2)\bbs_j|^4)^{1/4}\leq\frac{M}{n^2v^6}.$$
Replacing $\zeta_{kj2}$ with $n^{-1}tr\bbA_{kj}^{-1}(z_1)\bbD
\underline{\bbA}_{kj}^{-1}(z_2)\bbT$ in the above similarly gives
$$
E\Big(\sqrt{|\beta_{kj}(z_1)|}\xi_{kj}(z_1)\|\underline{\bbA}_{kj}^{-1}(z_2)\bbs_j\|n^{-1}tr\bbA_{kj}^{-1}(z_1)\bbD
\underline{\bbA}_{kj}^{-1}(z_2)\bbT\Big)^2\leq\frac{M}{nv^5}.
$$
Consequently
$
E|\frac{1}{n}\sum\limits_{j\neq
k}^{n}(E_j-E_{j-1})b_{12}(z_1)b_{12}(z_2)(\delta_{16})|^2\leq
\frac{M}{n^2v^6}.
$
This estimate apparently applies to the term involving $\delta_{17}$. Also, for $\delta_{18}$ we similarly have
$$
E|\delta_{18}|^2\leq \frac{M}{v^3} E\Big(\sqrt{|\underline{\beta_{kj}}(z_2)|}\xi_{kj}(z_1)\xi_{kj}(z_1)\|\underline{\bbA}_{kj}^{-1}(z_2)\bbs_j\|\Big)^2
\leq\frac{M}{n^2v^6}.
$$

The terms $\delta_2$ and $\delta_3$ can be similarly, even simpler,
proved to have the same order. Thus Lemma \ref{lem7} is complete.

\end{proof}

Combining (\ref{f63}), (\ref{f48}), (\ref{g5}), (\ref{g6}) and Lemma
\ref{lem7} with $\bbD=I$ we conclude that
\begin{eqnarray}
&&-\frac{1}{4\pi^2}\sum\limits_{k=1}^nE_{k-1}[Y_k(x_1)Y_k(x_2)] \non
&=&-\frac{1}{2h^2\pi^2}\int\int K'(\frac{x_1-z_1}{h})
K'(\frac{x_2-z_2}{h})a_{n1}(z_1,z_2)dz_1dz_2+o_p(1),\label{f48*}
\end{eqnarray}
where
$$
a_{n1}(z_1,z_2)=\frac{b_1(z_1)b_1(z_2)}{n^2}\sum\limits_{k=1}^nE\Big[tr\bbT
\bbA_k^{-1}(z_1)\bbT E_k(\bbA_k^{-1}(z_2))\Big].
$$
Thus, it is enough to investigate the uniform convergence of
$a_{n1}(z_1,z_2)$.

\subsection{The limit of $a_{n1}(z_1,z_2)$ }

Before developing the limit of $a_{n1}(z_1,z_2)$, we first establish
Lemma \ref{lem6} below, which improves (\ref{a39}) in Lemma
\ref{lem1}.
\begin{lemma}\label{lem6}
\begin{equation}\label{g26}
\frac{1}{n^2}E|tr\bbA_{k}^{-1}(z)\bbF^{-1}(z)\bbD-Etr\bbA_{k}^{-1}(z)\bbF^{-1}(z)\bbD|^2
\leq\frac{M}{n^2v^3\|\bbD\|^2},
\end{equation}
where $\bbF^{-1}(z)=(E\underline{m}_n\bbT+\bbI)^{-1}$ and $\bbD=\bbT$ or $(z\bbI-\frac{n-1}{n}b_{12}(z)\bbT)^{-1}$.
\end{lemma}
\begin{proof}
With notation
$\bbF_1^{-1}(z)=(z\bbI-\frac{n-1}{n}b_{12}(z)\bbT)^{-1}$, we start
the proof of Lemma \ref{lem6} by presenting the equality (2.9) in
\cite{b2}
\begin{equation}\label{f21}
\bbA_k^{-1}(z)=-\bbF_1^{-1}(z)+b_{12}(z)B(z)+C(z)+D(z),
\end{equation}
where
$$
B(z)=\sum\limits_{j\neq
k}\bbF_1^{-1}(z)(\bbs_j\bbs_j^T-n^{-1}\bbT)\bbA^{-1}_{kj}(z),
$$
$$
C(z)=\sum\limits_{j\neq
k}(\beta_{kj}(z)-b_{12}(z))\bbF_1^{-1}(z)\bbs_j\bbs_j^T\bbA^{-1}_{kj}(z)
$$
and
$$
D(z)=n^{-1}b_{12}(z)\bbF_1^{-1}(z)\bbT\sum\limits_{j\neq
k}(\bbA^{-1}_{kj}(z)-\bbA^{-1}_{k}(z)).
$$
By (\ref{f11}) we have
\begin{equation}\label{f1}
\frac{1}{n}tr\bbF_1^{-2}(z)\bbF_1^{-2}(\bar z)\leq M,\quad
\frac{1}{n}tr\bbF^{-2}(z)\bbF^{-2}(\bar z)\leq M.
\end{equation}
However sometimes we also use the fact that (see (2.10) in \cite{b2}
and Lemma 2.11 of \cite{b4})
\begin{equation}\label{g11}
\|\bbF_1(z)\|\leq\frac{M}{v}, \quad \|\bbF(z)\|\leq\frac{M}{v}.
\end{equation}

With $\bbH=\bbF^{-1}(z)$ or $\bbH=\bbI$, we first apply (\ref{f21})
with $z$ replaced by $z_2$ to prove that
\begin{equation}\label{f22}
\frac{1}{n}E\Big[tr\bbA_{k}^{-1}(z_2)\bbH\bbT\bbF_1^{-1}(z_1)\bbT\Big]=-\frac{1}{n}tr\bbF_1^{-1}(z_2)\bbH\bbT\bbF_1^{-1}(z_1)\bbT+O(\frac{1}{nv^{5/2}}),
\end{equation}
which, together with (\ref{f1}), then implies that
\begin{equation}
\label{g10}\Big|\frac{1}{n}E\Big[tr\bbA_{k}^{-1}(z)\bbH\bbT\bbF_1^{-1}(z)\bbT\Big]\Big|\leq
M.
\end{equation}

To see (\ref{f22}), first note that
$$
\frac{1}{n}E\Big[trB(z_2)\bbH\bbT\bbF_1^{-1}(z_1)\bbT\Big]=0.
$$
Second, applying (\ref{f13}) yields
$$
\frac{1}{n}E\Big[trC(z_2)\bbH\bbT\bbF_1^{-1}(z_1)\bbT\Big]=-\frac{1}{n}\sum\limits_{j\neq
k}E\Big[b_{12}^2(z_1)\eta_{kj}(z_1)\eta_{kj1}\Big]
$$
$$
-\frac{1}{n}\sum\limits_{j\neq
k}E\Big[b_{12}^2(z_1)\frac{1}{n}(tr\bbA^{-1}_{kj}(z)\bbT-Etr\bbA^{-1}_{kj}(z)\bbT)\frac{1}{n}tr\bbA^{-1}_{kj}(z)\bbH\bbT\bbF_1^{-1}(z_1)\bbT\bbF_1^{-1}(z)\bbT\Big]
$$
$$+\frac{1}{n}\sum\limits_{j\neq
k}E\Big[\beta_{kj}(z)b^2_{12}(z)\xi^2_{kj}\bbs_j^T\bbA^{-1}_{kj}(z)\bbH\bbT\bbF_1^{-1}(z_1)\bbT\bbF_1^{-1}(z)\bbs_j\Big],
$$
where
$$
\eta_{kj1}=\bbs_j^T\bbA^{-1}_{kj}(z)\bbH\bbT\bbF_1^{-1}(z)\bbT\bbF_1^{-1}(z)\bbs_j-\frac{1}{n}tr\bbA^{-1}_{kj}(z)\bbH\bbT\bbF_1^{-1}(z)\bbT\bbF_1^{-1}(z)\bbT.
$$
By Lemma \ref{lem8}, (\ref{f1}) and (\ref{g11})
\begin{equation}\label{g12}
E|\eta_{kj1}|^4\leq \frac{M}{n^2v^8}
(\frac{1}{n}tr\bbF_1^{-2}(z)\bbF_1^{-2}(\bar z))^2\leq
\frac{M}{n^2v^8},
\end{equation}
and via Holder's inequality and (\ref{f1})
\begin{equation}\label{g14}
\Big|\frac{1}{n}tr\bbA^{-1}_{kj}(z)\bbH\bbT\bbF_1^{-1}(z)\bbT\bbF_1^{-1}(z)\bbT\Big|\leq\frac{M}{v}.
\end{equation}
Appealing to (\ref{g12}) and (\ref{g9}) yields
$$
\Big|E\Big[b_{12}(z_1)\xi_{kj}(z_1)\eta_{kj1}\Big]\Big|\leq\frac{M}{nv^{5/2}}.
$$
We obtain from (\ref{g14}) and (\ref{g9})
$$
\Big|
E\Big[\frac{1}{n}(tr\bbA^{-1}_{kj}(z)\bbT-Etr\bbA^{-1}_{kj}(z)\bbT)\frac{1}{n}tr\bbA^{-1}_{kj}(z)\bbH\bbT\bbF_1^{-1}(z_1)\bbT\bbF_1^{-1}(z)\bbT\Big]\Big|
\leq\frac{M}{nv^{5/2}}.
$$
In view of (\ref{g12}), (\ref{g9}), (\ref{g14}) and Lemma \ref{lem1}
we have
$$
\Big|E\Big[\beta_{kj}(z_1)\xi^2_{kj}\bbs_j^T\bbA^{-1}_{kj}(z)\bbH\bbT\bbF_1^{-1}(z_1)\bbT\bbF_1^{-1}(z_1)\bbs_j\Big]\Big|
$$
$$
\leq
(E|\xi_{kj}(z_1)|^4)^{1/2}(E|\beta_{kj}(z)|^4E|\bbs_j^T\bbA^{-1}_{kj}(z)\bbH\bbT\bbF_1^{-1}(z_1)\bbT\bbF_1^{-1}(z)\bbs_j|^4)^{1/4}\leq\frac{M}{nv^{5/2}}.
$$

Let $\eta_{kj2}$ equal to
$$\bbs_j^T\bbA^{-1}_{kj}(z)\bbH\bbT\bbF_1^{-1}(z)\bbT\bbF_1^{-1}(z)\bbA^{-1}_{kj}(z)\bbs_j-\frac{1}{n}tr\bbA^{-1}_{kj}(z)\bbH\bbT\bbF_1^{-1}(z)\bbT\bbF_1^{-1}(z)\bbA^{-1}_{kj}(z)\bbT.
$$
Then, as in (\ref{g12}),
\begin{equation}
E|\eta_{kj2}|^4\leq \frac{M}{n^2v^{12}},\
\Big|\frac{1}{n}tr\bbA^{-1}_{kj}(z)\bbH\bbT\bbF_1^{-1}(z)\bbT\bbF_1^{-1}(z)\bbA^{-1}_{kj}(z)\bbT\Big|\leq\frac{M}{v^2}.
\end{equation}
This, together with (\ref{f13}), (\ref{g9}) and Lemma \ref{lem1},
ensures that
$$
|\frac{1}{n}E\Big[trD(z_2)\bbH\bbT\bbF_1^{-1}(z_1)\bbT\Big]|\leq\frac{M}{n^{3/2}v^{5/2}}.
$$
Thus (\ref{f22}) is true.

When replacing $\bbF_1^{-1}(z),\bbT$ with $\bbF^{-1}(\bar z),\bbI$,
respectively, (\ref{g10}) further ensures that
\begin{equation}
\label{g24} E\frac{1}{n}tr\bbF^{-1}(\bar z)\bbA_{k}^{-1}(\bar
z)\bbA_{k}^{-1}(z)\bbF^{-1}(z)=\frac{1}{v}\Im
\Big(E\frac{1}{n}tr\bbF^{-1}(\bar
z)\bbA_{k}^{-1}(z)\bbF^{-1}(z)\Big)\leq\frac{M}{v}.
\end{equation}

As before, we get a martingale representation as follows:
$$
\frac{1}{n}tr\bbA_{k}^{-1}(z)\bbF^{-1}(z)\bbD-Etr\bbA_{k}^{-1}(z)\bbF^{-1}(z)\bbD
$$$$=\frac{1}{n}\sum\limits_{j\neq k}(E_j-E_{j-1})\Big(tr\bbA_{k}^{-1}(z)\bbF^{-1}(z)\bbD-tr\bbA_{kj}^{-1}(z)\bbF^{-1}(z)\bbD\Big)
$$
$$
=\frac{1}{n}\sum\limits_{j\neq
k}(E_j-E_{j-1})\bbs_j^T\bbA_{kj}^{-1}(z)\bbF^{-1}(z)\bbD\bbA_{kj}^{-1}(z)\bbs_j\beta_{kj}
$$
$$
=\frac{b_{12}(z)}{n}\sum\limits_{j\neq
k}(E_j-E_{j-1})\Big(\phi_1+\phi_2+\phi_3\Big),
$$
where
$$
\phi_1=\bbs_j^T\bbA_{kj}^{-1}(z)\bbF^{-1}(z)\bbD\bbA_{kj}^{-1}(z)\bbs_j-\frac{1}{n}tr\bbA_{kj}^{-1}(z)\bbF^{-1}(z)\bbD\bbA_{kj}^{-1}(z)\bbT,
$$
$$
\phi_2=-b_{12}(z)\bbs_j^T\bbA_{kj}^{-1}(z)\bbF^{-1}(z)\bbD\bbA_{kj}^{-1}(z)\bbs_j\xi_{kj}(z),
$$
and
$$
\phi_3=b_{12}(z)\beta_{kj}\bbs_j^T\bbA_{kj}^{-1}(z)\bbF^{-1}(z)\bbD\bbA_{kj}^{-1}(z)\bbs_j\xi^2_{kj}(z).
$$
 Here in the last step one uses (\ref{f13}). By Lemma \ref{lem8}
and (\ref{g24}) we have
\begin{eqnarray}\label{g32}
E|\phi_1|^2&\leq&
\frac{M}{n^2}Etr\bbA_{kj}^{-1}(z)\bbF^{-1}(z)\bbD\bbA_{kj}^{-1}(z)\bbT\bbA_{kj}^{-1}(\bar
z)\bbD\bbF^{-1}(\bar z)\bbA_{kj}^{-1}(\bar z)\non
           &\leq&\frac{M}{n^2v^2\|\bbD\|^2}Etr\bbF^{-1}(\bar z)\bbA_{kj}^{-1}(\bar
           z)\bbA_{kj}^{-1}(z)\bbF^{-1}(z)\leq\frac{M}{nv^3\|\bbD\|^2}.
\end{eqnarray}
Similarly, from (\ref{f1}) and (\ref{f15}) it is easy to verify that
\begin{equation}\label{g34}
E|\phi_1|^8\leq \frac{M}{n^4v^{14}\|\bbD\|^8}.
\end{equation}
As in (\ref{g1}), by Lemma \ref{lem1}, (\ref{g9}) and (\ref{g24}) we
have
$$
E|\frac{1}{n}tr\bbA_{kj}^{-1}(z)\bbF^{-1}(z)\bbD\bbA_{kj}^{-1}(z)\bbT\eta_{kj}(z)|^2
$$
$$
\leq
\frac{M}{n}E\Big(|\frac{1}{n}tr\bbA_{kj}^{-1}(z)\bbF^{-1}(z)\bbD\bbA_{kj}^{-1}(z)\bbT|^2
\frac{1}{n}tr\bbA_{kj}^{-1}(z)\bbA_{kj}^{-1}(\bar z)\Big)
$$$$
\leq
\frac{M}{n\|\bbD\|^2}E\Big[\frac{1}{n}tr\bbA_{kj}^{-1}(z)\bbF^{-1}(z)\bbF^{-1}(\bar
z)\bbA_{kj}^{-1}(\bar z)
\Big(\frac{1}{n}tr\bbA_{kj}^{-1}(z)\bbA_{kj}^{-1}(\bar
z)\Big)^2\Big]
$$$$
\leq\frac{M}{n^3v^2\|\bbD\|^2}E\Big[tr\bbA_{kj}^{-1}(z)\bbF^{-1}(z)\bbF^{-1}(\bar
z)\bbA_{kj}^{-1}(\bar z)
$$
\begin{equation}\label{g35}
\qquad\qquad\times\Big(\Big|tr\bbA_{kj}^{-1}(z)-Etr\bbA_{kj}^{-1}(z)\Big|^2+|Etr\bbA_{kj}^{-1}(z)|^2\Big)\Big]
\leq\frac{M}{nv^3\|\bbD\|^2}.
\end{equation}
Note that
$$
(E_j-E_{j-1})\Big[\frac{1}{n}tr\bbA_{kj}^{-1}(z)\bbF^{-1}(z)\bbD\bbA_{kj}^{-1}(z)\bbT(\frac{1}{n}tr\bbA_{kj}^{-1}(z)\bbT-\frac{1}{n}E\bbA_{kj}^{-1}(z)\bbT)\Big]=0.
$$
This, together with (\ref{g34}), (\ref{g35}) and (\ref{g9}), implies
that
$$
E|\frac{1}{n}\sum\limits_{j\neq k}(E_j-E_{j-1})\phi_2|^2\leq
\frac{M}{n^2v^3\|\bbD\|^2}.
$$
By (\ref{g34}) and Lemma \ref{lem1} we
obtain
\begin{equation}\label{h3}
(E|\xi_{kj}(z)|^8)^{1/2}(E|\phi_1|^8E|\beta_{kj}|^8)^{1/4}\leq\frac{M}{n^3v^{13/2}}(E|\beta_{kj}|^4)^{1/4}\leq\frac{M}{n^3v^{13/2}}.
\end{equation}
From Holder's inequality, (\ref{f15}) and (\ref{f1})
$$
E|\frac{1}{n}tr\bbA_{kj}^{-1}(z)\bbF^{-1}(z)\bbD\bbA_{kj}^{-1}(z)\bbT|^8\leq E|\frac{1}{n}tr\bbA_{kj}^{-2}(z)\bbA_{kj}^{-2}(\bar z)\frac{1}{n}
tr\bbF^{-1}(z)\bbF^{-1}(\bar z)|^4\leq\frac{M}{v^{12}}.
$$
This, together with Lemma \ref{lem1} and (\ref{g9}), implies that
$$(E|\frac{1}{n}(tr\bbA_{kj}^{-1}(z)-Etr\bbA_{kj}^{-1}(z))|^8)^{1/2}(E|\frac{1}{n}tr\bbA_{kj}^{-1}(z)\bbF^{-1}(z)\bbD\bbA_{kj}^{-1}(z)\bbT|^8E|\beta_{kj}|^8)^{1/4}
$$$$\leq\frac{M}{n^4v^{10}}
$$
and that
$$
(E|\beta_{kj}|^4)^{1/2}(E\Big|\frac{1}{n}tr\bbA_{kj}^{-1}(z)\bbF^{-1}(z)\bbD\bbA_{kj}^{-1}(z)\bbT|^4|
\eta_{kj}(z)|^8\Big|)^{1/2}
$$$$\leq\frac{M}{n^2v^2}(E\Big|\frac{1}{n}tr\bbA_{kj}^{-1}(z)\bbF^{-1}(z)\bbD\bbA_{kj}^{-1}(z)\bbT
\frac{1}{n}tr\bbA_{kj}^{-1}(z)\Big|^4)^{1/2}\leq\frac{M}{n^2v^5}.
$$
These and (\ref{h3}) ensure that
$$
E|\phi_3|^2\leq\frac{M}{n^2v^5}.
$$
Thus, Lemma \ref{lem6} is complete.

\end{proof}

 To later use, we now consider a more general form than
$a_{n1}(z_1,z_2)$
 \begin{equation}
a_{n2}^{(1)}(z_1,z_2)=\frac{1}{n^2}\sum\limits_{k=1}^nE\Big[tr\bbT
\bbA_k^{-1}(z_1)\bbT E_k(\bbA_k^{-1}(z_2))\bbH\Big].\label{f28}
\end{equation}
One should note that $a_{n2}^{(1)}(z_1,z_2)$ reduces to
$a_{n2}(z_1,z_2)$ when $\bbH=\bbI$.

Applying the definition of $C(z_1)$ and (\ref{f13})
gives
\begin{equation}\label{h5}
\frac{1}{n}E\Big[tr\bbT C(z_1)\bbT
E_k(\bbA_k^{-1}(z_2))\bbH\Big]=C_1(z_1)+C_2(z_1),\end{equation}
 where
$$
C_1(z_1)=-b_{12}^2(z_1)\frac{1}{n}\sum\limits_{j\neq
k}E\Big[\xi_{kj}(z_1)\bbs_j^T\hat{\bbA}^{-1}_{kjk}(z_1,z_2)\bbs_j\Big]
$$
and
$$
C_2(z_1)=b_{12}^2(z_1)\frac{1}{n}\sum\limits_{j\neq
k}E\Big[\beta_{kj}(z_1)\xi^2_{kj}(z)\bbs_j^T\hat{\bbA}^{-1}_{kjk}(z_1,z_2)\bbs_j\Big].
$$
Here
\begin{equation}\label{h4}
 \hat{\bbA}^{-1}_{kjk}(z_1,z_2)=\bbA^{-1}_{kj}(z_1)\bbT
E_k(\bbA_{k}^{-1}(z_2))\bbH\bbT\bbF_1^{-1}(z_1).
\end{equation}
 Define
 $$
\zeta_{kj3}=\bbs_j^T\hat{\bbA}^{-1}_{kjk}(z_1,z_2)\bbs_j-\frac{1}{n}tr\hat{\bbA}^{-1}_{kjk}(z_1,z_2)\bbT.
 $$
Consider $j>k$ first. Then by Lemma \ref{lem8} and (\ref{f1})
 \begin{equation}\label{g16}
E|\zeta_{kj3}|^4\leq\frac{M}{n^2v^8},
 \end{equation}
 and via an argument similar to (\ref{g8}), (\ref{f15}) and Holder's inequality
 \begin{equation}\label{g17}
 E|\frac{1}{n}tr\hat{\bbA}^{-1}_{kjk}(z_1,z_2)\bbT|^4\leq\frac{M}{v^6}.
 \end{equation}
It follows from (\ref{g16}), (\ref{g17}) and (\ref{g9}) that
$$
|C_2(z_1)|\leq \frac{M}{n}\sum\limits_{j\neq
k}(E|\xi_{kj}(z_1)|^4)^{1/2}\Big[E|\beta_{kj}(z_1)|^4
$$$$\qquad\qquad\times\Big(E|\zeta_{kj3}|^4+E|\frac{1}{n}tr\hat{\bbA}^{-1}_{kjk}(z_1,z_2)\bbT|^4\Big)\Big]^{1/4}\leq\frac{M}{nv^{5/2}}.
$$
As for $C_1(z_1)$, by (\ref{g16}), (\ref{g9}), (\ref{f10}) and
(\ref{f5}) with $\bbD=\bbH\bbT\bbF_1^{-1}(z_1)\bbT$ we have
$$\frac{1}{n}\sum\limits_{j>
k}E\Big[\xi_{kj}(z_1)\bbs_j^T\hat{\bbA}^{-1}_{kjk}(z_1,z_2)\bbs_j\Big]
=\frac{1}{n}\sum\limits_{j> k}E\Big[\eta_{kj}(z_1)\zeta_{kj3}
$$$$+\frac{1}{n^2}(tr\bbA_{kj}^{-1}-Etr\bbA_{kj}^{-1})(tr\hat{\bbA}^{-1}_{kjk}(z_1,z_2)\bbT-
Etr\hat{\bbA}^{-1}_{kjk}(z_1,z_2)\bbT)\Big]=O(\frac{1}{nv^{5/2}}).
$$

When $j<k$, decompose $\bbA_{k}^{-1}(z_2)$ as
$$\bbA_{kj}^{-1}(z_2)-\bbA_{kj}^{-1}(z_2)\bbs_j\bbs_j^T\bbA_{k}^{-1}(z_2)\beta_{kj}.$$
Then, apparently, the above argument for the case $j>k$ also works if we replace
$E_k(\bbA_k^{-1}(z_2))$ in $\hat{\bbA}^{-1}_{kjk}(z_1,z_2)$ with $E_k(\bbA_{kj}^{-1}(z_2))$.  For another term of $C_2(z_1)$ by
(\ref{f7}), (\ref{g9}) and Lemma \ref{lem1}
$$
E\Big|\beta_{kj}(z_1)\underline{\beta}_{kj}(z_2)\xi^2_{kj}(z)\bbs_j^T\bbA_{kj}^{-1}(z_1)\bbT\underline{\bbA}_{kj}^{-1}(z_2)\bbs_j\bbs_j^T
\underline{\bbA}_{kj}^{-1}(z_2)\bbH\bbT\bbF_1^{-1}(z_1)\bbs_j\Big|
$$
$$
\leq\frac{M}{v}(E|\beta_{kj}(z_1)|^2E|\underline{\beta}_{kj}(z_2)|^2E|\xi_{kj}|^8E|\bbs_j^T\underline{\bbA}_{kj}^{-1}(z_2)\bbH\bbT\bbF_1^{-1}(z_1)\bbs_j|^4)^{1/4}
\leq\frac{M}{nv^{5/2}},$$ because \begin{equation}
\label{g25}E|\zeta_{kj4}|^4\leq\frac{M}{n^2v^4},\quad
E|\frac{1}{n}tr\underline{\bbA}_{kj}^{-1}(z_2)\bbH\bbT\bbF_1^{-1}(z_1)\bbT|^4\leq\frac{M}{v^2},
\end{equation}
with
$\zeta_{kj4}=\bbs_j^T\underline{\bbA}_{kj}^{-1}(z_2)\bbH\bbT\bbF_1^{-1}(z_1)\bbs_j-n^{-1}tr\underline{\bbA}_{kj}^{-1}(z_2)\bbH\bbT\bbF_1^{-1}(z_1)\bbT$.
As for another term of $C_1(z_1)$, an application of (\ref{f13})
yields
$$
\frac{1}{n}\sum\limits_{j<
k}E\Big[\xi_{kj}(z_1)\bbs_j^T\bbA^{-1}_{kj}(z_1)\bbT
E_k\Big(\bbA_{kj}^{-1}(z_2)\bbs_j\bbs_j^T\bbA_{kj}^{-1}(z_2)\beta_{kj}(z_2)\Big)\bbH\bbT\bbF_1^{-1}(z_1)\bbs_j\Big]
$$
\begin{eqnarray}
&=&\frac{1}{n}\sum\limits_{j<
k}E\Big[\underline{\beta}_{kj}(z_2)\xi_{kj}(z_1)\bbs_j^T\bbA^{-1}_{kj}(z_1)\bbT
\underline{\bbA}_{kj}^{-1}(z_2)\bbs_j\bbs_j^T\underline{\bbA}_{kj}^{-1}(z_2)\bbH\bbT\bbF_1^{-1}(z_1)\bbs_j\Big]\non
&\ &\quad =\frac{b_{12}}{n}\sum\limits_{j<
k}[C_{11}+C_{12}+C_{13}+C_{14}+C_{15}+C_{16}],\label{g29}
\end{eqnarray}
where
$$
C_{11}=E\Big[\xi_{kj}(z_1)\zeta_{kj1}\zeta_{kj4}\Big],\
C_{12}=E\Big[\xi_{kj}(z_1)\zeta_{kj1}n^{-1}tr\underline{\bbA}_{kj}^{-1}(z_2)\bbH\bbT\bbF_1^{-1}(z_1)\bbT\Big]
$$
\begin{eqnarray*}
C_{13}&=&E\Big[\xi_{kj}(z_1)\zeta_{kj4}n^{-1}tr\bbA^{-1}_{kj}(z_1)\bbT
\underline{\bbA}_{kj}^{-1}(z_2)\bbT\Big],\\
C_{14}&=&E\Big[\xi_{kj}(z_1)n^{-1}tr\bbA^{-1}_{kj}(z_1)\bbT
\underline{\bbA}_{kj}^{-1}(z_2)\bbT
n^{-1}tr\underline{\bbA}_{kj}^{-1}(z_2)\bbH\bbT\bbF_1^{-1}(z_1)\bbT\Big]\\
&=&\frac{1}{n}E\Big[(tr\bbA_{kj}^{-1}(z_1)\\
&&-Etr\bbA_{kj}^{-1}(z_1))n^{-1}tr\bbA^{-1}_{kj}(z_1)\bbT
\underline{\bbA}_{kj}^{-1}(z_2)\bbT
n^{-1}tr\underline{\bbA}_{kj}^{-1}(z_2)\bbH\bbT\bbF_1^{-1}(z_1)\bbT\Big]\\
C_{15}&=&-E\Big[\underline{\beta}_{kj}(z_2)\underline{\xi}_{kj}(z_2)\xi_{kj}(z_1)\bbs_j^T\bbA^{-1}_{kj}(z_1)\bbT
\underline{\bbA}_{kj}^{-1}(z_2)\bbs_j\zeta_{kj4}\Big]
\end{eqnarray*}
and
$$
C_{16}=-E\Big[\underline{\beta}_{kj}(z_2)\underline{\xi}_{kj}(z_2)\xi_{kj}(z_1)
 \bbs_j^T\bbA^{-1}_{kj}(z_1)\bbT
\underline{\bbA}_{kj}^{-1}(z_2)\bbs_jn^{-1}tr\underline{\bbA}_{kj}^{-1}(z_2)\bbH\bbT\bbF_1^{-1}(z_1)\bbT\Big].
$$
Appealing to (\ref{g8}), (\ref{f16}), (\ref{g9}) and (\ref{g25})
yields
$$
|C_{1j}|\leq\frac{M}{nv^{5/2}}, \quad j=1,2,3.
$$
By (\ref{g10}), (\ref{g8}), (\ref{g9}) and (\ref{g26}) we obtain $|C_{14}|\leq\frac{M}{nv^{5/2}}$. We conclude from
(\ref{f7}), (\ref{g9}), (\ref{g23}), Lemma \ref{lem1} and
(\ref{g25}) that
$$
|C_{15}|\leq
\frac{M}{\sqrt{v}}E\Big|\sqrt{|\underline{\beta}_{kj}(z_2)|}\underline{\xi}_{kj}(z_2)\xi_{kj}(z_1)\|\bbs_j^T\bbA^{-1}_{kj}(z_1)\|\zeta_{kj4}\Big|
\qquad \qquad \qquad \qquad$$
$$
\qquad\leq
\frac{M}{\sqrt{v}}\Big(E(\sqrt{|\underline{\beta}_{kj}(z_2)}|\|\bbs_j^T\bbA^{-1}_{kj}(z_1)\|)^4E|\underline{\xi}_{kj}(z_2)|^4E|\xi_{kj}(z_1)|^4E|\zeta_{kj4}|^4\Big)^{1/4}
$$
$$
\leq\frac{M}{n^{3/2}v^{5/2}}.\qquad \qquad \qquad \qquad\qquad \qquad \qquad \qquad\qquad \qquad \qquad
$$
Similarly
$$
|C_{16}|\leq\frac{M}{nv^{5/2}}.
$$
Summarizing the above we have proved that
\begin{equation}
\Big|\frac{1}{n}E\Big[tr\bbT C(z_1)\bbT
E_k(\bbA_k^{-1}(z_2))\bbH\Big]\Big|\leq\frac{M}{nv^{5/2}}.
\end{equation}

Consider $D(z_1)$ now. When $j>k$ using (\ref{f13}) and recalling the definition of $ \hat{\bbA}^{-1}_{kjk}(z_1,z_2)$ in (\ref{h4}) we obtain
$$
\frac{1}{n}E\Big[tr \bbT D(z_1)\bbT E_k(\bbA_k^{-1}(z_2))\bbH\Big]
=\frac{1}{n^2}b_{12}(z_1)\sum\limits_{j\neq k}[D_1+D_2+D_3]
$$
where
$$
D_1=-\frac{1}{n}E\Big[tr\hat{\bbA}^{-1}_{kjk}(z_1,z_2)\bbT\bbA^{-1}_{kj}(z_1)\bbT\Big],\
D_2=E\Big[\zeta_{kj5}\xi_{kj}(z_1)\beta_{kj}(z_1)\Big]
$$
and
$$
D_3=\frac{1}{n}E\Big[tr \hat{\bbA}^{-1}_{kjk}(z_1,z_2)\bbT\bbA^{-1}_{kj}(z_1)\bbT\xi_{kj}(z_1)\beta_{kj}(z_1)\Big].
$$
with
$$
\zeta_{kj5}=\bbs_j^T \hat{\bbA}^{-1}_{kjk}(z_1,z_2)\bbT\bbA^{-1}_{kj}(z_1)\bbs_j-\frac{1}{n} tr \hat{\bbA}^{-1}_{kjk}(z_1,z_2)\bbT\bbA^{-1}_{kj}(z_1)\bbT,
$$
Using (\ref{f1}), (\ref{f10}) and Holder's inequality
\begin{equation}\label{g27}
\frac{1}{n}E\Big|tr \hat{\bbA}^{-1}_{kjk}(z_1,z_2)\bbT\bbA^{-1}_{kj}(z_1)\bbT\Big|^2\leq\frac{M}{v^5}.
\end{equation}
By Lemma 2.7 \cite{b4} and (\ref{f1})
\begin{equation}\label{g28}
E|\zeta_{kj5}|^2\leq\frac{M}{nv^6}.
\end{equation}
Thus
$$
|D_1|\leq\frac{M}{v^{5/2}},\  |D_2|\leq \frac{M}{nv^\frac{7}{2}}, \
|D_3|\leq\frac{M}{\sqrt{n}v^3}.
$$
Hence when $j>k$
 $$
\Big|\frac{1}{n}E\Big[tr \bbT D(z_1)\bbT
E_k(\bbA_k^{-1}(z_2))\bbH\Big]\Big|\leq\frac{M}{nv^{5/2}}.
 $$

When $j<k$, divide $\bbA_k^{-1}(z_2)$ into the sum:
$$\bbA_{kj}^{-1}(z_2)
-\bbA_{kj}^{-1}(z_2)\bbs_j\bbs_j^T\bbA_{kj}^{-1}(z_2)\beta_{kj}(z_2).$$
Apparently, the above argument for the case $j>k$ also works for the
term involving $E_k(\bbA_{kj}^{-1}(z_2))$ if we replace
$E_k(\bbA_k^{-1}(z_2))$ with $E_k(\bbA_{kj}^{-1}(z_2))$. Another
term is
$$
\frac{b_{12}(z_1)}{n^2}\sum\limits_{j\neq
k}E\Big[\beta_{kj}(z_1)\underline{\beta}_{kj}(z_2)\bbs_j^T\bbA_{kj}^{-1}(z_1)\bbT\underline{\bbA}_{kj}^{-1}(z_2)\bbs_j
$$$$
\qquad\times\bbs_j^T\underline{\bbA}_{kj}^{-1}(z_2)\bbH\bbT\bbF_1^{-1}(z_1)\bbT\bbA_{kj}^{-1}(z_1)\bbs_j\Big],
$$
which has, via ((\ref{f7}), (\ref{g19}) and (\ref{g22})
with $\bbD=\bbH\bbT\bbF_1^{-1}(z_1)\bbT$, an order of
$\frac{1}{nv^{5/2}}$.

Thus, the contribution from $C(z_1)$ and $D(z_1)$ is negligible.

Next consider $B(z_1)$. It follows from (\ref{f6}) that
\begin{eqnarray}
&&\frac{1}{n}E\Big[tr \bbT B(z_1)\bbT E_k(\bbA_k^{-1}(z_2))\bbH\Big]\label{f19}
\\ &=&\frac{1}{n}\sum\limits_{j<k}E\Big[\bbs_j^T
\hat{\bbA}^{-1}_{kjk}(z_1,z_2)\bbs_j-n^{-1}tr\bbT
\hat{\bbA}^{-1}_{kjk}(z_1,z_2)\Big]=B_1(z_1)+B_2(z_1),\nonumber
\end{eqnarray}
where
$$
B_1(z_1)=-\frac{1}{n}\sum\limits_{j<k}E\Big[\underline{\beta}_{kj}(z_2)\bbs_j^T\bbA_{kj}^{-1}(z_1)\bbT\underline{\bbA}_{kj}^{-1}(z_2)\bbs_j\bbs_j^T\underline{\bbA}_{kj}^{-1}(z_2)\bbH\bbT\bbF_1^{-1}(z_1)\bbs_j\Big]
$$
and
$$
B_2(z_1)=-\frac{1}{n^2}\sum\limits_{j<k}E\Big[\underline{\beta}_{kj}(z_2)\bbs_j^T\underline{\bbA}_{kj}^{-1}(z_2)\bbH\bbT\bbF_1^{-1}(z_1)\bbA_{kj}^{-1}(z_1)\bbT\underline{\bbA}_{kj}^{-1}(z_2)\bbs_j\Big].
$$
For $B_2(z_1)$ we further write
$$
B_2(z_1)=B_{21}(z_1)+B_{22}(z_1),
$$
where
$$
B_{21}(z_1)=-\frac{b_{12}(z_2)}{n^2}\sum\limits_{j<k}E\Big[\frac{1}{n}tr\underline{\bbA}_{kj}^{-1}(z_2)\bbH\bbT\bbF_1^{-1}(z_1)\bbA_{kj}^{-1}(z_1)\bbT\underline{\bbA}_{kj}^{-1}(z_2)\bbT\Big],
$$
and
\begin{eqnarray*}
B_{22}(z_1)&=&\frac{b_{12}(z_2)}{n^2}\sum\limits_{j<k}E\Big[\underline{\beta}_{kj}(z_2)\underline{\xi}_{kj}(z_2)\\
&& \qquad \times \bbs_j^T\underline{\bbA}_{kj}^{-1}(z_2)\bbH\bbT\bbF_1^{-1}(z_1)\bbA_{kj}^{-1}(z_1)\bbT\underline{\bbA}_{kj}^{-1}(z_2)\bbs_j\Big].
\end{eqnarray*}
The inequality similar to (\ref{g27}) ensures that
$|B_{21}(z_1)|\leq\frac{M}{nv^{5/2}}$, while
$|B_{22}(z_1)|\leq\frac{M}{n^{3/2}v^{3}}$ by estimates similar to
(\ref{g27}) and (\ref{g28}). Therefore $B_{2}(z_1)$ is negligible.

By (\ref{f13}), (\ref{f16}) and the estimates of $C_{1j},j=1,2,3,4$
in (\ref{g29}) we have
$$
\Big|B_1(z_1)+\frac{b_{12}(z_2)}{n}\sum\limits_{j<k}E\Big[\bbs_j^T\bbA_{kj}^{-1}(z_1)\bbT\underline{\bbA}_{kj}^{-1}(z_2)\bbs_j\bbs_j^T\underline{\bbA}_{kj}^{-1}(z_2)\bbH\bbT\bbF_1^{-1}(z_1)\bbs_j\Big]\Big|
$$$$=O(\frac{M}{nv^{5/2}}).
$$
In the mean time, (\ref{f16}) and (\ref{g25}) ensure that
$$
\frac{1}{n}\sum\limits_{j<k}E\Big[\bbs_j^T\bbA_{kj}^{-1}(z_1)\bbT\underline{\bbA}_{kj}^{-1}(z_2)\bbs_j\bbs_j^T\underline{\bbA}_{kj}^{-1}(z_2)\bbH\bbT\bbF_1^{-1}(z_1)\bbs_j\Big]
$$
$$
=\frac{1}{n^3}\sum\limits_{j<k}E\Big[tr\bbT\bbA_{kj}^{-1}(z_1)\bbT\underline{\bbA}_{kj}^{-1}(z_2)tr\underline{\bbA}_{kj}^{-1}(z_2)\bbH\bbT\bbF_1^{-1}(z_1)\bbT\Big]
+O(\frac{M}{nv^{5/2}}).
$$
Furthermore we apply (\ref{f5}), (\ref{g26}), (\ref{f10}), (\ref{f13}) and (\ref{g32}) to obtain
$$
\frac{1}{n^2}E\Big[tr\bbT\bbA_{kj}^{-1}(z_1)\bbT\underline{\bbA}_{kj}^{-1}(z_2)tr\underline{\bbA}_{kj}^{-1}(z_2)\bbH\bbT\bbF_1^{-1}(z_1)\bbT\Big]
$$$$=
\frac{1}{n^2}E\Big[tr\bbT\bbA_{kj}^{-1}(z_1)\bbT\underline{\bbA}_{kj}^{-1}(z_2)\Big]E\Big[tr\underline{\bbA}_{kj}^{-1}(z_2)\bbH\bbT\bbF_1^{-1}(z_1)\bbT\Big]
+O(\frac{M}{n^2v^{5}}).
$$
In addition, by (\ref{f10}), (\ref{g19}), (\ref{g22}), (\ref{f13}) and (\ref{g10}) we have
$$
\frac{1}{n^2}E\Big[tr\bbT\bbA_{kj}^{-1}(z_1)\bbT\underline{\bbA}_{kj}^{-1}(z_2)\Big]E\Big[tr\underline{\bbA}_{kj}^{-1}(z_2)\bbH\bbT\bbF_1^{-1}(z_1)\bbT\Big]
$$
$$
=\frac{1}{n^2}E\Big[tr\bbT\bbA_{k}^{-1}(z_1)\bbT\underline{\bbA}_{k}^{-1}(z_2)\Big]E\Big[tr\underline{\bbA}_{k}^{-1}(z_2)\bbH\bbT\bbF_1^{-1}(z_1)\bbT\Big]+
O(\frac{1}{nv^{2}}).
$$
It follows that
\begin{eqnarray}
&&\Big|B_1(z_1)+\frac{j-1}{n^3}b_{12}(z_2)E\Big(tr\bbT\bbA_{k}^{-1}(z_1)\bbT\underline{\bbA}_{k}^{-1}(z_2)\Big)\label{f20}\\
&&\times E\Big(tr\bbA_{k}^{-1}(z_2)\bbH\bbT\bbF_1^{-1}(z_1)\bbT\Big)\Big|
\leq\frac{M}{nv^{5/2}}. \nonumber
\end{eqnarray}

Summarizing the argument from (\ref{h5}) to (\ref{f20}) yields
\begin{equation}\label{f30}
\frac{1}{n}E\Big[tr\bbT \bbA_k^{-1}(z_1)\bbT
E_k(\bbA_k^{-1}(z_2))\bbH\Big]=-\frac{1}{n}E\Big[tr\bbA_{k}^{-1}(z_2)\bbH\bbT\bbF_1^{-1}(z_1)\bbT\Big]
\end{equation}
\begin{eqnarray*}
&&-\Big[\frac{j-1}{n^3}b_{12}(z_1)b_{12}(z_2)E\Big(tr\bbT\bbA_{k}^{-1}(z_1)\bbT\underline{\bbA}_{k}^{-1}(z_2)\Big)E\Big(tr\bbA_{k}^{-1}(z_2)\bbH\bbT\bbF_1^{-1}(z_1)\bbT\Big)\Big]
\\
&&+O(\frac{1}{nv^{5/2}}).
\end{eqnarray*}
When $\bbH=\bbI$,  (\ref{f30}) and (\ref{f22}) produce
$$
\frac{1}{n}E\Big[tr\bbT \bbA_k^{-1}(z_1)\bbT
E_k(\bbA_k^{-1}(z_2))\Big]\Big[1-\frac{j-1}{n}b_{12}(z_1)b_{12}(z_2)\frac{1}{n}tr\bbF_1^{-1}(z_2)\bbT\bbF_1^{-1}(z_1)\bbT\Big]
$$
\begin{equation}\label{f31}
=\frac{1}{n}tr\bbF_1^{-1}(z_2)\bbT\bbF_1^{-1}(z_1)\bbT+O(\frac{1}{nv^{5/2}}).
\end{equation}

By the formula ( see (2.2) of \cite{s3})
$$
\underline{m}_n(z)=-\frac{1}{zn}\sum\limits_{k=1}^n\beta_k(z)
$$
we have
\begin{equation}\label{d12}
E\beta_1(z)=-zE\underline{m}_n(z)
\end{equation}
It follows from (\ref{b22}) that
$$
|E\beta_1(z)-b_1(z)|=|b_1(z)^2E(\beta_1(z)\xi_1^2(z))|\leq\frac{M}{nv}
$$
and from (\ref{f10}) that
$$
|b_1(z)-b_{12}(z)|\leq\frac{M}{nv}.
$$
These, together with (\ref{f23}), imply that
\begin{equation}\label{g36}
|b_{12}(z)-\underline{m}_n^0(z)|\leq\frac{M}{nv^{3/2}}.
\end{equation}
This, along with (\ref{f31}), (\ref{f15}) and (\ref{f10}), yields that
 \begin{eqnarray}\label{f33}
&&\frac{1}{n}E\Big[tr\bbT \bbA_k^{-1}(z_1)\bbT
E_k(\bbA_k^{-1}(z_2))\Big]\times\Big[1-\frac{j-1}{n}b_n(z_1,z_2)\Big]
\non
&=&\frac{c_nb_n(z_1,z_2)}{z_1z_2\underline{m}_n^0(z_1)\underline{m}_n^0(z_2)}+O(\frac{1}{nv^{5/2}}),
\end{eqnarray}
where
$$
b_n(z_1,z_2)=c_n\underline{m}_n^0(z_1)\underline{m}_n^0(z_2)\int\frac{t^2dH_n(t)}{(1+t\underline{m}_n^0(z_1))
(1+t\underline{m}_n^0(z_2))}.
$$
It follows that
\begin{equation}
\label{f24}
a_{n1}(z_1,z_2)=b_n(z_1,z_2)\frac{1}{n}\sum\limits_{j=1}^n\frac{1}{1-\frac{j-1}{n}b_n(z_1,z_2)}+O(\frac{1}{nv^{5/2}}).
\end{equation}

From (2.19) in \cite{b2} and the inequality above (6.37) in
\cite{ker} we see that
\begin{equation}
\label{f56} |1-\frac{j-1}{n}b_n(z_1,z_2)|\geq Mv, \quad
|1-tb_n(z_1,z_2)|\geq Mv,\quad\text{for any} \quad t\in [0,1].
\end{equation}
It follows that
$$
|\frac{1}{n}\sum\limits_{j=1}^n\frac{1}{1-\frac{j-1}{n}b_n(z_1,z_2)}-\int^1_0\frac{1}{1-tb_n(z_1,z_2)}|\leq\frac{M}{nv^2}.
$$
This ensures that
$$
a_{n1}(z_1,z_2)=b_n(z_1,z_2)\int^1_0\frac{1}{1-tb_n(z_1,z_2)}+O(\frac{1}{nv^{5/2}})
$$$$
=-\log (1-b_n(z_1,z_2))+O(\frac{1}{nv^{5/2}})
$$
\begin{equation}\label{g30}
=-\log \Big((z_1-z_2)\underline{m}_n^0(z_1)\underline{m}_n^0(z_2)\Big)-\log (\underline{m}_n^0(z_1)-\underline{m}_n^0(z_2))+O(\frac{1}{nv^{5/2}}),
\end{equation}
where in the last step one uses the fact that by (\ref{f51})
$$
z_1-z_2=\frac{\underline{m}_n^0(z_1)-\underline{m}_n^0(z_2)}{\underline{m}_n^0(z_1)\underline{m}_n^0(z_2)}(1-b_n(z_1,z_2)).
$$

So far we have considered $z\in\gamma_2$. The above argument evidently works for the case of $z\in\gamma_1$ due to symmetry. To deal with the cases when $z$ belongs to two vertical lines of the contour, we need the estimates (1.9a) and (1.9b) of \cite{b2}, which hold under our truncation level.
That is
\begin{equation}
P(\|\bbA\|\geq \mu_1)=o(n^{-l}),\ P(\lambda_{\min}^\bbA\leq \mu_2)=o(n^{-l}),
\end{equation}
for any $\mu_1>\limsup \bbT(1+\sqrt{c})^2$,
$\mu_2<\liminf\bbT(1-\sqrt{c})^2$ and $l$. This implies that
\begin{equation}
P(\|\bbA_k\|\geq \mu_1)=o(n^{-l}),\ P(\lambda_{\min}^{\bbA_k}\leq \mu_2)=o(n^{-l}).\label{h8}
\end{equation}
Let $B=\bigcap\limits_{k=1}^n B_k$ where $B_k=(a_l-\eta<\lambda_{\min}^{\bbA_k}<\|\bbA_k\|<a_r-\eta)$ with $\eta>0$ so that $a_r-\eta>\limsup \bbT(1+\sqrt{c})^2$ and $a_l-\eta<\liminf\bbT(1-\sqrt{c})^2$. Also define $B_{n+1}=(a_l-\eta<\lambda_{\min}^{\bbA}<\|\bbA\|<a_r-\eta)$ and let $C_k=B_k\cap B_{n+1}$.  It follows that on the two vertical lines $\gamma_2\cup \gamma_4$ (\ref{h6}) is equal to
\begin{eqnarray}\label{h7}
&=&-\frac{1}{h}\frac{1}{2\pi i}\sum\limits_{k=1}^n(E_k-E_{k-1})\int
K'(\frac{x-z}{h})\log \beta_k(z)dz\non
&=&-\frac{1}{h}\frac{1}{2\pi i}\sum\limits_{k=1}^n(E_k-E_{k-1})\int
K'(\frac{x-z}{h})\log \beta_k(z)I(C_k)dz+o_p(1).
\end{eqnarray}

We then introduce $\hat{\beta}_k(z)$, a truncated version of $\beta_k(z)$. Select a sequence of positive numbers $\varepsilon_n$
satisfying for some $\beta\in(0,1)$,
\begin{equation}\label{g21}
\varepsilon_n\downarrow0,\ \ \varepsilon_n\geq n^{-\beta}.
\end{equation}
Define
$$
\gamma_l=\{a_l+iv: v\in[n^{-1}\varepsilon_n,v_0h]\cup [-v_0h,-n^{-1}\varepsilon_n]\}
$$
and
$$
\gamma_r=\{a_r+iv:v\in[n^{-1}\varepsilon_n,v_0h]\cup [-v_0h,-n^{-1}\varepsilon_n]\}.
$$
Write $\gamma_n=\gamma_r\cup\gamma_l$. We
can now define the process
\begin{equation}\label{g20}
\hat{\beta}_k(z)=\begin{cases} \beta_k(z),& \text{if}\ z\in
\gamma_n\\
\frac{nv+\varepsilon_n}{2\varepsilon_n}\beta_k(z_{r1})+\frac{\varepsilon_n-nv}{2\varepsilon_n}\beta_k(z_{r2}),&
\text{if}\  u=a_r,v\in[-n^{-1}\varepsilon_n,n^{-1}\varepsilon_n],\\
\frac{nv+\varepsilon_n}{2\varepsilon_n}\beta_k(z_{l1})+\frac{\varepsilon_n-nv}{2\varepsilon_n}\beta_k(z_{l2}),&
\text{if}\  u=a_l,v\in[-n^{-1}\varepsilon_n,n^{-1}\varepsilon_n],
\end{cases}
\end{equation}
where $z_{r1}=a_r+in^{-1}\varepsilon_n,
z_{r2}=a_r-in^{-1}\varepsilon_n, z_{l1}=a_l+in^{-1}\varepsilon_n,
z_{l2}=a_l-in^{-1}\varepsilon_n$.
Note that $\|(\bbA_k-z\bbI)^{-1}I(C_k)\|\leq \frac{1}{\eta}$, $\|(\bbA-z\bbI)^{-1}I(C_k)\|\leq \frac{1}{\eta}$ and then $|\beta_k(z)I(C_k)|=|1-\bbs_k^T(\bbA-z\bbI)^{-1}\bbs_kI(C_k)|\leq M\bbs_k^T\bbs_k$. It follows that
\begin{equation}\label{h10}
P(|Q_k|\geq 1/2)\leq\frac{2\varepsilon_nE(\bbs_k^T\bbs_k)}{n}\rightarrow 0,
\end{equation}
where $Q_k=\beta_k(z)(1/\beta_k(z)-1/\hat{\beta}_k(z))I(C_k).$
By (\ref{h9}) and (\ref{g3}) we obtain
$$
\Big|\frac{1}{h}\frac{1}{2\pi i}\sum\limits_{k=1}^n(E_k-E_{k-1})\oint
K'(\frac{x-z}{h})(\log \beta_k(z)-\log\hat{\beta}_k(z))I(C_k)I(|Q_k|<\frac{1}{2})dz\Big|
$$
$$
\leq \frac{M\varepsilon_n^2}{n^2}\sum\limits_{k=1}^n(\bbs_k^T\bbs_k)^4\stackrel{i.p.}\longrightarrow 0.
$$
This, together with (\ref{h7}) and (\ref{h10}), ensures that
$$
(\ref{h6})=\frac{1}{h}\frac{1}{2\pi i}\sum\limits_{k=1}^n(E_k-E_{k-1})\int
K'(\frac{x-z}{h})\log \hat{\beta}_k(z)I(C_k)dz+o_p(1)
$$
$$
=\frac{1}{h}\frac{1}{2\pi i}\sum\limits_{k=1}^n(E_k-E_{k-1})\int
K'(\frac{x-z}{h})\log \Big(\frac{\hat{\beta}^{tr}_k(z)}{\hat{\beta}_k(z)}\Big)dz+o_p(1),
$$
where $\hat{\beta}^{tr}_k(z)$ is similarly defined according to $\hat{\beta}_k(z)$. Moreover, for the truncation versions, the
higher moments of $\bbA^{-1}(z)$, $\bbA_k^{-1}(z)$ and
$\bbA_{kj}^{-1}(z)$ are bounded (see (3.1) in \cite{b2}). Also, as
pointed out in the paragraph below (3.2) in \cite{b2}, the moments
of $\beta_1(z),\beta_{12}(z), \beta^{tr}(z),
s_1^T\bbA_1^{-1}(z_1)\bbT\bbA_1^{-1}(z_2)\bbs_1$ are bounded as
well. Using these facts, all the estimates holding for
$z\in\gamma_1\cup \gamma_2$ also holds for the case where
$z=z_{r1},z_{r2}$ or $z\in \gamma_r\cup\gamma_l$.  Via these facts,
the arguments of the cases $z=z_{r1},z_{r2}$ or $z\in
\gamma_r\cup\gamma_l$, two vertical lines, can follow from those of the
case $z\in\gamma_1\cup \gamma_2$ (here we omit the details) and
hence their limits have the same form as (\ref{g30}).

In the mean time, appealing to Cauchy's theorem gives
\begin{equation}\label{g53}
\frac{1}{h^2}\oint_{\mathcal{C}_1}\oint_{\mathcal{C}_2}
K'(\frac{x_1-z_1}{h})K'(\frac{x_2-z_2}{h})\log\Big((z_1-z_2)\underline{m}_n^0(z_1)\underline{m}_n^0(z_2)\Big)dz_1dz_2=0,
\end{equation}
where the contour $\mathcal{C}_2$ is also a rectangle formed with four
vertices $a_l-\varepsilon\pm 2iv_0h$ and $a_{r}+\varepsilon\pm
2iv_0h$ with $\varepsilon>0$. One should note that the contour
$\mathcal{C}_2$ encloses the contour $\mathcal{C}_1$.
Thus, in view of (\ref{g30}), it remains to find the limit of the
following
\begin{equation}\label{f62}
-\frac{1}{2h^2\pi^2}\oint_{\mathcal{C}_1}\oint_{\mathcal{C}_2}
K'(\frac{x_1-z_1}{h})K'(\frac{x_2-z_2}{h})\log(\underline{m}_n^0(z_1)-\underline{m}_n^0(z_2))dz_1dz_2,
\end{equation}
which is done in Appendix 3.

\section{The limit of mean function}

The aim in the section is to find the limit of
$$
\frac{1}{2\pi i}\oint K(\frac{x-z}{h})n(Em_n(z)-m_n^0(z))dz.
$$
It is thus sufficient to investigate the uniform convergence
$nh(E\underline {m}_n(z)-\underline{m}_n^0(z))$ on the contour.

Recall that $ \bbF^{-1}(z)=(E\underline{m}_n\bbT+\bbI)^{-1}$ and
then write ( see (5.2) in \cite{b4})
 \begin{equation}\label{b15}
n(c_n\int\frac{dH_n(t)}{1+tE\underline{m}_n}+zc_nE(m_n(z)))=nD_n,
 \end{equation}
 where
 $$
D_n=E\beta_1\Big[\bbs_1^T\bbA_1^{-1}(z)\bbF^{-1}(z)\bbs_1-\frac{1}{n}E\Big(tr\bbF^{-1}(z)\bbT\bbA^{-1}(z)\Big)\Big].
$$
It follows that (see (3.20) in \cite{b4})
\begin{equation}
\label{a32}
n(E\underline{m}_n(z)-\underline{m}_n^0(z))=-\frac{n\underline{m}_n^0(z)D_n}{1-c_nE\underline{m}_n\underline{m}_n^0\int\frac{t^2dH_n(t)}{(1+tE\underline{m}_n)(1+t\underline{m}_n^0)}}.
\end{equation}


Considered $z\in\gamma_1\cup\gamma_2$ first. Applying (\ref{b22}) and (\ref{b23}) yields
\begin{eqnarray}\label{b28}
&&E\Big(tr\bbF^{-1}(z)\bbT\bbA_1^{-1}(z)\Big)-E\Big(tr\bbF^{-1}(z)\bbT\bbA^{-1}(z)\Big)\non
&=&E\Big(\beta_1\bbs_1^T\bbA_1^{-1}(z)\bbF^{-1}(z)\bbT\bbA_1^{-1}(z)\bbs_1\Big)
\non
&=&b_1E\Big([1-b_1\xi_1+b_1\beta_1\xi_1^2(z)]\bbs_1^T\bbA_1^{-1}(z)\bbF^{-1}(z)\bbT\bbA_1^{-1}(z)\bbs_1\Big)\non
&=&b_1E\frac{1}{n}tr
\bbA_1^{-1}(z)\bbF^{-1}(z)\bbT\bbA_1^{-1}(z)\bbT-d_{n1}+d_{n2}+d_{n3}\non
&=&b_1E\frac{1}{n}tr
\bbA_1^{-1}(z)\bbF^{-1}(z)\bbT\bbA_1^{-1}(z)\bbT+O(\frac{1}{nv^{5/2}}),
\end{eqnarray}
where
$$
d_{n1}=b_1^2E\Big[\eta_1(z)(\bbs_1^T\bbA_1^{-1}(z)\bbF^{-1}(z)\bbT\bbA_1^{-1}(z)\bbs_1-\frac{1}{n}tr
\bbA_1^{-1}(z)\bbF^{-1}(z)\bbT\bbA_1^{-1}(z))\Big]
$$\begin{eqnarray*}
d_{n2}&=&\frac{b_1^2}{n}E\Big[\Big(tr\bbA^{-1}(z)\bbT-Etr\bbA^{-1}(z)\bbT\Big)\bbs_1^T\bbA_1^{-1}(z)\bbF^{-1}(z)\bbT\bbA_1^{-1}(z)\bbs_1\Big]\non
&=&\frac{b_1^2}{n^2}E\Big[\Big(tr\bbA^{-1}(z)\bbT-Etr\bbA^{-1}(z)\bbT\Big)\non
&\qquad&\times\Big(tr\bbA_1^{-1}(z)\bbF^{-1}(z)\bbT\bbA_1^{-1}(z)\bbT-Etr\bbA_1^{-1}(z)\bbF^{-1}(z)\bbT\bbA_1^{-1}(z)\bbT\Big)\Big]
\end{eqnarray*}
and
$$
d_{n3}=b_1E\Big[\beta_1\xi_1^2(z)\bbs_1^T\bbA_1^{-1}(z)\bbF^{-1}(z)\bbT\bbA_1^{-1}(z)\bbs_1\Big]
$$
$$
=b_1E\Big[\beta_1\xi_1^2(z)(\bbs_1^T\bbA_1^{-1}(z)\bbF^{-1}(z)\bbT\bbA_1^{-1}(z)\bbs_1-\frac{1}{n}tr\bbA_1^{-1}(z)\bbF^{-1}(z)\bbT\bbA_1^{-1}(z)\bbT)\Big]
$$
$$
+b_1E\Big[\beta_1\xi_1^2(z)\frac{1}{n}tr\bbA_1^{-1}(z)\bbF^{-1}(z)\bbT\bbA_1^{-1}(z)\bbT)\Big].
$$
It follows from (\ref{g32}) and Lemma \ref{lem1} that
$$
|d_{nj}|\leq\frac{M}{nv^{5/2}},\quad j=1,3,
$$
where we also use the fact that
\begin{eqnarray}
&&|\frac{1}{n}tr\bbA_1^{-1}(z)\bbF^{-1}(z)\bbT\bbA_1^{-1}(z)\bbT)|\non
&\leq&
\frac{M}{n}\Big[tr\bbF^{-1}(z)\bbF^{-1}(\bar{z})tr(\bbA_1^{-1}(z)\bbA_1^{-1}(\bar
z))^2\Big]^{1/2}\leq\frac{M}{v^{3/2}}\label{b24}.
\end{eqnarray}
While, Lemma \ref{lem1} and an estimate similar to (\ref{f5}) yield
$|d_{n2}|\leq\frac{M}{nv^{5/2}}$.

Next by (\ref{b22}) \begin{eqnarray}
&&nE\Big[\beta_1\bbs_1^T\bbA_1^{-1}(z)\bbF^{-1}(z)\bbs_1\Big]-E(\beta_1)E\Big(tr\bbF^{-1}(z)\bbT\bbA_1^{-1}(z)\Big)\non
&=&-nb_1^2E\Big[\xi_1\bbs_1^T\bbA_1^{-1}(z)\bbF^{-1}(z)\bbs_1\Big]+nb^2_1E\Big[\beta_1\xi_1^2\bbs_1^T\bbA_1^{-1}(z)\bbF^{-1}(z)\bbs_1\Big]\non
&&
-b^2_1E(\beta_1\xi_1^2)E\Big[tr\bbA_1^{-1}(z)\bbF^{-1}(z)\bbT\Big]\non
&=&f_{n1}+f_{n2}+f_{n3}+f_{n4},\label{f27}
\end{eqnarray}
where
$$
f_{n1}=-nb_1^2E\Big[\eta_1(\bbs_1^T\bbA_1^{-1}(z)\bbF^{-1}(z)\bbs_1-\frac{1}{n}tr\bbA_1^{-1}(z)\bbF^{-1}(z)\bbT)\Big],
$$
\begin{eqnarray*}
f_{n2}&=&-b_1^2E\Big[\Big(tr\bbA^{-1}(z)\bbT-Etr\bbA^{-1}(z)\bbT\Big)\bbs_1^T\bbA_1^{-1}(z)\bbF^{-1}(z)\bbs_1\Big]\non
&=&\frac{b_1^2}{n}E\Big[\Big(tr\bbA^{-1}(z)\bbT-Etr\bbA^{-1}(z)\bbT\Big)\non
&\qquad&\times\Big(tr\bbA_1^{-1}(z)\bbF^{-1}(z)\bbT-Etr\bbA_1^{-1}(z)\bbF^{-1}(z)\bbT\Big)\Big],
\end{eqnarray*}
$$
f_{n3}=nb^2_1E\Big[\beta_1\xi_1^2\Big(\bbs_1^T\bbA_1^{-1}(z)\bbF^{-1}(z)\bbs_1-\frac{1}{n}tr\bbA_1^{-1}(z)\bbF^{-1}(z)\bbT\Big)\Big]
$$
and
$$
f_{n4}=b_1^2\Big(E\Big[\beta_1\xi_1^2tr\bbA_1^{-1}(z)\bbF^{-1}(z)\bbT\Big]-E(\beta_1\xi_1^2)E\Big[tr\bbA_1^{-1}(z)\bbF^{-1}(z)\bbT\Big]\Big).
$$
We conclude from (\ref{g26}) and Lemma \ref{lem1} that
$$
\sqrt{h}|f_{n2}|\leq\frac{M}{nv^{5/2}}
$$
and that
$$
\sqrt{h}|f_{n3}|\leq\frac{M}{\sqrt{nv^{2}}},
$$
because by (\ref{g24}) and Lemma \ref{lem8}
$$
E|\bbs_1^T\bbA_1^{-1}(z)\bbF^{-1}(z)\bbs_1-\frac{1}{n}tr\bbA_1^{-1}(z)\bbF^{-1}(z)\bbT|^2\leq\frac{M}{nv}.
$$
It follows from Holder's inequality and (\ref{g26}) that
\begin{eqnarray*}
|f_{n4}|&\leq&
M(E|\beta_1|^4E|\xi_1|^4)^{1/4}(E|tr\bbA_1^{-1}(z)\bbF^{-1}(z)\bbT-Etr\bbA_1^{-1}(z)\bbF^{-1}(z)\bbT|^2)^{1/2}\\
&\leq&\frac{M}{nv^{5/2}}.
\end{eqnarray*}

Therefore from (\ref{b15}), (\ref{b28}), (\ref{f27}) (\ref{f49}) and
(\ref{i1}) we obtain
\begin{eqnarray*}
n\sqrt{h}D_n&=&b_1^2E\frac{\sqrt{h}}{n}tr
\bbA_1^{-1}(z)\bbF^{-1}(z)\bbT\bbA_1^{-1}(z)\bbT+f_{n1}+O(\frac{1}{\sqrt{nv^{5/2}}})\non
&=&-b_1^2E\frac{\sqrt{h}}{n}tr
\bbA_1^{-1}(z)\bbF^{-1}(z)\bbT\bbA_1^{-1}(z)\bbT-b_1^2\sqrt{h}\frac{EX_{11}^4-3}{n}\non
\end{eqnarray*}
$$
\times\sum\limits_{k=1}^pE\Big[
(\bbT^{1/2}\bbA_1^{-1}(z)\bbT^{1/2})_{kk}(\bbT^{1/2}\bbA_1^{-1}(z)\bbF^{-1}(z)\bbT^{1/2})_{kk}\Big]+O(\frac{1}{\sqrt{nv^{5/2}}})
$$
\begin{equation}\label{f38}
 =-b_1^2E\frac{\sqrt{h}}{n}tr
\bbA_1^{-1}(z)\bbF^{-1}(z)\bbT\bbA_1^{-1}(z)\bbT+O(\frac{1}{\sqrt{nv^{5/2}}}).
\end{equation}

A careful inspection on the argument leading to (\ref{f30})
indicates that it also works for
$E\frac{1}{n}tr\bbA_1^{-1}(z)\bbF^{-1}(z)\bbT\bbA_1^{-1}(z)\bbT$ and
the main difference is that treating the latter does not need to
distinguish between the cases $j<k$ and $j>k$. Thus, applying (\ref{f30})
with $\bbH=\bbF^{-1}(z),z_1=z_2=z$ and replacing $(j-1)/n$ there
with one we have
\begin{equation}\label{f32}
\frac{1}{n}E\Big[tr\bbT\bbA_1^{-1}(z)\bbF^{-1}(z)\bbT\bbA_1^{-1}(z)\Big]=-\frac{1}{n}E\Big[tr\bbA_{1}^{-1}(z)\bbF^{-1}(z)\bbT\bbF_1^{-1}(z)\bbT\Big]-
\end{equation}$$\frac{b_{12}(z)b_{12}(z)}{n^2}E\Big(tr\bbT\bbA_{1}^{-1}(z)\bbT\underline{\bbA}_{1}^{-1}(z)\Big)E\Big(tr\bbA_{1}^{-1}(z)\bbF^{-1}(z)\bbT\bbF_1^{-1}(z)\bbT\Big)
+O(\frac{1}{nv^{5/2}}).
$$

We claim that
\begin{equation}\label{f23}
|E\underline{m}_n(z)-\underline{m}_n^0(z)|\leq\frac{M}{nv^{3/2}}
\end{equation}
so that (\ref{f33}) is applicable.  To prove (\ref{f23}), we first
show that
\begin{equation}
\label{f36}
|E\underline{m}_n(z)-\underline{m}_n^0(z)|\leq\frac{M}{nv^{2}}.
\end{equation}
Evidently, (\ref{f14}) yields
$$
|\frac{1}{n}E\Big(tr\bbT\bbA_{1}^{-1}(z)\bbT\bbA_{1}^{-1}(z)\Big)|\leq\frac{M}{v}.
$$
It follows from (\ref{g10}) and (\ref{f32}) that
\begin{equation}\label{f37}
|\frac{1}{n}E\Big[tr\bbT\bbA_1^{-1}(z)\bbF^{-1}(z)\bbT\bbA_1^{-1}(z)\Big]|\leq\frac{M}{v}.
\end{equation}
This, together with (\ref{f38}), ensures that
\begin{equation}\label{f39}
|\bbD_n|\leq\frac{M}{nv}.
\end{equation}
It is proved in \cite{ker} that (see (6.38) in \cite{ker})
$$
|1-c_n\underline{m}_n^0(z)E\underline{m}_n(z)\int\frac{t^2dH_n(t)}{(1+t\underline{m}_n^0(z))(1+tE\underline{m}_n(z))}|\geq
M_3v.
$$
Hence (\ref{f36}) follows from the above inequality, (\ref{f39}) and
(\ref{a32}). We then conclude from (\ref{d1}), (\ref{d2}) and
(\ref{f11}) that
\begin{equation}\label{f40}
|1-c_n\underline{m}_n^0(z)E\underline{m}_n(z)\int\frac{t^2dH_n(t)}{(1+t\underline{m}_n^0(z))(1+tE\underline{m}_n(z))}|\geq
M_2\sqrt{v},
\end{equation}
which, along with (\ref{f39}) and (\ref{a32}), immediately gives
$(\ref{f23})$.

We are now in a position to use (\ref{f33}) with $z_1=z_2=z$ and
replacing $(j-1)/n$ there with one so that
\begin{equation}
\frac{1}{n}E\Big[tr\bbT \bbA_1^{-1}(z)\bbT
\bbA_1^{-1}(z)\Big]=\frac{\frac{c_n}{z}\int\frac{t^2dH_n(t)}{(1+t\underline{m}_n^0(z))^2
}}{1-c_n\underline{m}_n^0(z)\underline{m}_n^0(z)\int\frac{t^2dH_n(t)}{(1+t\underline{m}_n^0(z))^2}}+O(\frac{1}{nv^{5/2}}).
\end{equation}
A direct application of (\ref{f22}) and (\ref{f23})  yields
\begin{equation}\label{f34}
\frac{1}{n}E\Big[tr\bbA_{1}^{-1}(z)\bbF^{-1}(z)\bbT\bbF_1^{-1}(z)\bbT\Big]=-\frac{c_n}{z^2}\int\frac{t^2}{(1+t\underline{m}_n^0(z))^3}+O(\frac{1}{nv^{5/2}}).
\end{equation}
It follows from (\ref{f32})-(\ref{f34}) and (\ref{g36}) that
\begin{equation}\label{f35}
\sqrt{h}\frac{1}{n}E\Big[tr\bbT\bbA_1^{-1}(z)\bbF^{-1}(z)\bbT\bbA_1^{-1}(z)\Big]=\sqrt{h}\frac{\frac{c_n}{z}\int\frac{t^2dH_n(t)}{(1+t\underline{m}_n^0(z))^3
}}{1-c_n(\underline{m}_n^0(z))^2\int\frac{t^2dH_n(t)}{(1+t\underline{m}_n^0(z))^2}}+O(\frac{1}{nv^{5/2}}).
\end{equation}

We then conclude from (\ref{a32}), (\ref{f38}), (\ref{f35}),
(\ref{d1}) and (\ref{f40}) that
\begin{equation}
nh(E\underline{m}_n(z)-\underline{m}_n^0(z))=h\frac{c_n(\underline{m}_n^0(z))^3\int\frac{t^2dH_n(t)}{(1+t\underline{m}_n^0(z))^3
}}{\Big(1-c_n(\underline{m}_n^0(z))^2\int\frac{t^2dH_n(t)}{(1+t\underline{m}_n^0(z))^2}\Big)^2}+O(\frac{1}{\sqrt{nv^{5/2}}}).
\end{equation}
The case when $z$ lies in the vertical lines on the contour can be
handled similarly as pointed out in the last section with the truncation version of $\beta_k(z)$ replaced with the truncation version of $n(Em_n(z)-m_n^0(z))$ (one may refer to \cite{b2} as well).

It remains to find the limit of the following
\begin{equation}\label{f50}
\frac{1}{4\pi i}\oint
K(\frac{x-z}{h})\frac{c_n(\underline{m}_n^0(z))^3\int\frac{t^2dH_n(t)}{(1+t\underline{m}_n^0(z))^3
}}{\Big(1-c_n(\underline{m}_n^0(z))^2\int\frac{t^2dH_n(t)}{(1+t\underline{m}_n^0(z))^2}\Big)^2}dz,
\end{equation}
which is done in Appendix 3.

\section{The proof of Theorem \ref{theo2}}

For any finite constants $l_1,\cdots,l_r$, by Cauchy's theorem and Fubini's theorem we write
\begin{equation}
\frac{n}{\sqrt{\ln \frac{1}{h}
}}\sum\limits_{j=1}^rl_j\Big(F_n(x_j)-\int^{x_j}_{-\infty}\frac{1}{h}\int
K(\frac{t-y}{h})dF^{c_n,H_n}(y)dt\Big)\label{g42}
\end{equation}
$$=\frac{n}{\sqrt{\ln \frac{1}{h}
}}\sum\limits_{j=1}^rl_j\Big(\int^{x_j}_{-\infty}f_n(t)dt-\int^{x_j}_{-\infty}\frac{1}{h}\int
K(\frac{t-y}{h})dF^{c_n,H_n}(y)dt\Big)
$$
$$
=-\frac{n}{2h\pi i\sqrt{\ln
\frac{1}{h}}}\sum\limits_{j=1}^rl_j(\int^{x_j}_{-\infty}\oint_{\mathcal{C}_1}
K(\frac{t-z}{h})(tr\bbA^{-1}(z)-ns_n(z))dzdt
$$
$$
=-\frac{n}{2h\pi i\sqrt{\ln
\frac{1}{h}}}\sum\limits_{j=1}^rl_j\oint_{\mathcal{C}_1}
\Big[\int^{x_j}_{-\infty}K(\frac{t-z}{h})dt\Big](tr\bbA^{-1}(z)-ns_n(z))dz,
$$
where the contour ${\mathcal{C}_1}$ is defined as before.

Furthermore, we conclude from (\ref{g39}) and integration by parts that
$$
\frac{1}{2h\pi i\sqrt{\ln \frac{1}{h}}}\oint_{\mathcal{C}_1}
\Big[\int^{x}_{-\infty}K(\frac{t-z}{h})dt\Big](tr\bbA^{-1}(z)-Etr\bbA^{-1}(z))dz
$$
$$
=-\frac{1}{2h\pi i\sqrt{\ln
\frac{1}{h}}}\sum\limits_{k=1}^n(E_k-E_{k-1})\oint_{\mathcal{C}_1}
\Big[\int^{x}_{-\infty}K(\frac{t-z}{h})dt\Big]\Big[\log
\beta_k(z)\Big]'dz
$$
\begin{equation}
=\frac{1}{2h\pi i\sqrt{\ln
\frac{1}{h}}}\sum\limits_{k=1}^n(E_k-E_{k-1})\oint_{\mathcal{C}_1}
K(\frac{x-z}{h})\log
\frac{\beta_k^{tr}(z)}{\beta_k(z)}dz,\label{g40}
\end{equation}
where in the last step one uses the fact that via (\ref{a25})
\begin{equation}\label{g47}
\Big[\int^{x}_{-\infty}K(\frac{t-z}{h})dt\Big]'=K(\frac{x-z}{h}).
\end{equation}
It is observed that the unique difference between (\ref{g40}) and
(\ref{b12}) is that the test function $K'(\frac{x-z}{h})$ there is
replaced by $K(\frac{x-z}{h})$. Therefore, repeating the arguments
in Section 2 we obtain that (\ref{g40}) is asymptotically normal with
covariance (see \eqref{f62})
\begin{equation}\label{g43}
-\frac{1}{2h^2\pi^2\ln
\frac{1}{h}}\oint_{\mathcal{C}_1}\oint_{\mathcal{C}_2}
K(\frac{x_1-z_1}{h})K(\frac{x_2-z_2}{h})\log(\underline{m}_n^0(z_1)-\underline{m}_n^0(z_2))dz_1dz_2.
\end{equation}

Also, for the nonrandom part we have
\begin{equation}\label{g48}
\frac{1}{2h\pi i\sqrt{\ln \frac{1}{h}}}\oint_{\mathcal{C}_1}
\Big[\int^{x}_{-\infty}K(\frac{t-z}{h})dt\Big]n(Etr\bbA^{-1}(z)-m_n^0(z))dz.
\end{equation}
Note that
$$
|\frac{1}{h}\int^{x}_{-\infty}K(\frac{t-z}{h})dt|<\infty.
$$
Thus, repeating the arguments in Section 3 we see that (\ref{g48})
becomes
\begin{equation}\label{g44}
\frac{1}{4h\pi i\sqrt{\ln \frac{1}{h}}}\oint
\Big[\int^{x}_{-\infty}K(\frac{t-z}{h})dt\Big]\frac{c_n(\underline{m}_n^0(z))^3\int\frac{t^2dH_n(t)}{(1+t\underline{m}_n^0(z))^3
}}{\Big(1-c_n(\underline{m}_n^0(z))^2\int\frac{t^2dH_n(t)}{(1+t\underline{m}_n^0(z))^2}\Big)^2}dz
\end{equation}
$$
+O(\frac{1}{nh^3\sqrt{\ln\frac{1}{h}}}).
$$
The limits of (\ref{g43}) and (\ref{g44}) are derived in Appendix 3.

\section{Appendix 1}

This Appendix collects some frequently used Lemmas.
\begin{lemma} When $z$ lies in the segments $\gamma_1\cup \gamma_2$,
\label{lem1} $$|\underline{m}_n^0(z)|\leq M, |Em_n(z)|\leq M,
|b_1(z)|\leq M,\ E|\beta_1(z)|^4\leq M,\ E|\beta_1^{tr}(z)|^4\leq
M$$ and
\begin{equation}\label{a39}
\frac{1}{n^8}E|tr\bbA^{-1}(z)\bbD-Etr\bbA^{-1}(z)\bbD|^8
\leq\frac{M}{n^8v^{12}\|\bbD\|^8},
\end{equation}
\begin{equation}\label{b18}
E|\eta_1(z)|^8\leq\frac{M}{n^4v^4},\quad  E|\xi_1(z)|^8\leq
\frac{M}{n^4v^4},
\end{equation}
where $\bbD$ is a non-random matrix with nonzero spectral norm.
\end{lemma}
\begin{remark}
Lemma \ref{lem6} in Section 2 improves (\ref{a39}) when $\|D\|$ is
not bounded above by a constant but $\frac{1}{n}tr\bbD\bbD^*\leq M$.
\end{remark}

\begin{proof} We remind readers that $z=u+iv$ with $v=Mh$ and $u\in [a,b]$
when $z$ lies in the segments $\gamma_1\cup \gamma_2$.

 As pointed out in (6.1) in \cite{ker}, we obtain
 \begin{equation}\label{b15}
|\underline{m}_n^0(z)|\leq M,\quad |m_n^0(z)|\leq M.
\end{equation}
From integration by parts and Theorem 3 in \cite{ker} we have for
$$
|Em_n(z)-m_n^0(z)|=|\int^{+\infty}_{-\infty}\frac{1}{x-z}d(EF^{\bbA_n}(x)-F^{c_n,H_n}(x))|
$$
$$
=|\int^{+\infty}_{-\infty}\frac{EF^{\bbA_n}(x)-F^{c_n,H_n}(x)}{(x-z)^2}dx|\leq\frac{\pi\sup\limits_{x}|EF^{\bbA_n}(x)-F^{c_n,H_n}(x)|}{v}\leq
M.
$$
This implies
 \begin{equation}\label{b16}
|Em_n(z)|\leq M, \quad  |E\underline{m}_n(z)|\leq M.
\end{equation}
It follows from Lemma 7 and lemma 8 in \cite{ker} that
$$
|b_1(z)|\leq M.
$$
Repeating the argument of Lemma 3 in \cite{ker} gives (\ref{a39}).

 Write
\begin{equation}\label{b17}
\beta_1^{tr}(z)=b_1(z)-\frac{1}{n}\beta_1^{tr}(z)b_1(z)(tr\bbA^{-1}(z)\bbT-Etr\bbA^{-1}(z)\bbT).
\end{equation}
We then conclude that
\begin{eqnarray}
E|\beta_1^{tr}(z)|^4&\leq&
M+\frac{M}{n^4v^4}E|tr\bbA^{-1}(z)\bbT-Etr\bbA^{-1}(z)\bbT|^4 \non
&\leq& M+\frac{M}{n^4v^{10}} \leq M.\label{g15}
\end{eqnarray}

Expand $\beta_1(z)$ as
$$
\beta_1(z)=\beta_1^{tr}(z)-\beta_1^{tr}(z)\beta_1(z)\eta_1(z).
$$
It follows from (\ref{g15}), (\ref{f44}) and Lemma \ref{lem8}
that
$$
E|\beta_1(z)|^4\leq
E|\beta_1^{tr}(z)|^4+\frac{1}{v^2}\Big(E|\beta_1(z)|^4E|\eta_1(z)\beta_1^{tr}(z)|^8\Big)^{1/2}
$$
$$
\leq
M+\frac{1}{n^2v^4}\Big(E|\beta_1(z)|^4E|\beta_1^{tr}(z)|^4\Big)^{1/2}\leq
M+\frac{1}{n^2v^4}\Big(E|\beta_1(z)|^4\Big)^{1/2}.
$$
Solving the inequality gives
$$
E|\beta_1(z)|^4\leq M.
$$

It follows from (\ref{f10}) and (\ref{b16}) that
\begin{equation}\label{c3}|\frac{1}{n}Etr\bbA_1^{-1}(z)|\leq M.
\end{equation}
By Lemma \ref{lem8} and (\ref{a39})
$$E|\eta_1(z)|^8\leq
\frac{M}{n^8}E(tr\bbA_1^{-1}(z)\bbT\bbA_1^{-1}(\bar
z)\bbT)^4\leq\frac{M}{n^4}E(tr\bbA_1^{-1}(z)\bbA_1^{-1}(\bar z))^4
$$
\begin{equation}\label{h2}
\leq\frac{M}{n^4v^4}E\Big[\Im\Big(tr\bbA_1^{-1}(z)-Etr\bbA_1^{-1}(z)\Big)\Big]^4+\frac{M}{n^4v^4}(\Im
Etr\bbA_1^{-1}(z))^4 \leq\frac{M}{n^4v^4},
\end{equation}
where $\bbA^{-1}(\bar z)$ denotes the complex conjugate of
$\bbA^{-1}(z)$ and we also use
$$
\frac{1}{n}Etr\bbA_1^{-1}(z)\bbA_1^{-1}(\bar{z})=\frac{1}{v}\Im(\frac{1}{n}Etr\bbA_1^{-1}(z)).
$$
This, together with (\ref{a39}), yields the estimate of $\xi_1(z)$.
\end{proof}

\begin{lemma}
\label{lem8} (Lemma 2.7 of \cite{b4}) Suppose that $X_1,\cdots,X_n$
are i.i.d real random variables with $EX_1=0$ and $EX_1^2=1$. Let
$\bbx=(X_1,\cdots,X_n)^T$ and $\bbD$ be any $n\times n$ complex
matrix. Then for any $p\geq2$
$$
E|\bbx^T\bbD\bbx-tr\bbD|^p\leq
M_p\Big[(E|X_1|^4tr\bbD\bbD^*)^{p/2}+E|X_1|^{2p}tr(\bbD\bbD^*)^{p/2}\Big].
$$
\end{lemma}

\section{Appendix 2}

This section is to verify Remark \ref{rem1} and Theorem \ref{rem2}.

We first prove that (\ref{d1}) is true when $\bbT$ becomes the
identity matrix. When $\bbT$ is the identity matrix, the left hand
of (\ref{d1}) becomes
$$
1-c_n\frac{(\underline{m}_n^0(z))^2}{(1+\underline{m}_n^0(z))^2}.
$$
In view of (\ref{a3}) we have
\begin{equation}\label{d7}
1-c_n\frac{(\underline{m}_n^0(z))^2}{(1+\underline{m}_n^0(z))^2}=1-\frac{1}{c_n}(1+z\underline{m}_n^0(z))^2
\end{equation}
and
\begin{equation}\label{d8}
\underline{m}_n^0(z)=\frac{-(z+1-c_n)+\sqrt{(z-1-c_n)^2-4c_n}}{2z}.
\end{equation}
Thus,
$$
1-c_n\frac{(\underline{m}_n^0(z))^2}{(1+\underline{m}_n^0(z))^2}=1-\frac{1}{c_n}\Big[\frac{-(z-1-c_n)+\sqrt{(z-1-c_n)^2-4c_n}}{2}\Big]^2
$$
$$
=\frac{1}{2c_n}\frac{\sqrt{4c_n-(z-1-c_n)^2}}{2c}\Big[(z-1-c_n)i+\sqrt{4c_n-(z-1-c_n)^2}\Big]
$$
$$
=\frac{\sqrt{(z-1-c_n)^2-4c_n}}{2c_n}\Big[(z-1-c_n)+\sqrt{(z-1-c_n)^2-4c_n}\Big]
$$
$$
=\frac{\sqrt{(z-1-c_n)^2-4c_n}}{c_n}(c_nzm_n^0(z)+z-1)
$$$$=\frac{\sqrt{(z-1-c_n)^2-4c_n}}{2c_n}\frac{(1+c_nm_n^0(z))}{c_nm_n^0(z)},
$$
where in the last two steps one uses the facts that
\begin{equation}\label{d9}
m_n^0(z)=-\frac{c_n+z-1-\sqrt{(z-1-c_n)^2-4c_n}}{c_nz}
\end{equation}
and
\begin{equation}
\label{d11}m_n^0(z)=\frac{1}{1-c_n-c_nzm_n^0(z)-z}.
\end{equation}
It follows from (\ref{d11}) and (\ref{b15}) that
\begin{equation}
|\frac{1}{1+c_nm_n^0(z)}|=|1-c_n-zc_nm_n^0(z)|\leq M.
\end{equation}

Write
$$
\sqrt{(z-1-c_n)^2-4c_n}=\sqrt{(a-z)(b-z)}.
$$
Then it is simple to verify that
$$
|\sqrt{(a-z)(b-z)}|\geq \sqrt{(b-a)v}.
$$
Thus, (\ref{d1}) is true when $\bbT$ is the identity matrix.

\begin{lemma}
 \label{lem3} Under the assumptions that $n^5h^{29/2}\leq M$ and that
 \begin{equation}\label{d2}
\int\frac{dH_n(t)}{|1+t\underline{m}_n^0(z)|^{20}}<\infty,
\end{equation}
 (\ref{f11}) is true and
\begin{equation}\label{f9}
\int\frac{dH_n(t)}{|z-\frac{n-1}{n}tb_{12}(z)|^{4}}<\infty.
\end{equation}
\end{lemma}
\begin{remark}
The assumptions that $n^5h^{29/2}\leq M$ and (\ref{d2}) are only
used in this Lemma.
\end{remark}

\proof First recall that (see (6.30) in \cite{ker})
\begin{equation}\label{d3}
|E\underline{m}_n(z)-\underline{m}_n^0(z)|\leq\frac{M}{nv^{5/2}}.
\end{equation}
It is straightforward to verify that
\begin{equation}
\label{d6}
|1+tE\underline{m}_n(z)|^{2}-|1+t\underline{m}_n^0(z)|^{2}\leq
M|E\underline{m}_n(z)-\underline{m}_n^0(z)|.
\end{equation}
We then write
\begin{eqnarray}
&&\int\frac{dH_n(t)}{|1+tE\underline{m}_n(z)|^{4}}-\int\frac{dH_n(t)}{|1+t\underline{m}_n^0(z)|^{4}}\non
&=&-\int\frac{|1+tE\underline{m}_n(z)|^{2}-|1+t\underline{m}_n^0(z)|^{2}dH_n(t)}
{|1+tE\underline{m}_n(z)|^{4}|1+t\underline{m}_n^0(z)|^{2}}\non
&&\qquad\qquad\qquad-\int\frac{|1+tE\underline{m}_n(z)|^{2}-|1+t\underline{m}_n^0(z)|^{2}dH_n(t)}{|1+tE\underline{m}_n(z)|^{2}|1+t\underline{m}_n^0(z)|^{4}}.\label{d5}
\end{eqnarray}
Obviously, for the above last term, by Holder's inequality,
(\ref{d2}), (\ref{d3}) and (\ref{d6})
\begin{eqnarray}\label{d4}
&&|\int\frac{|1+tE\underline{m}_n(z)|^{2}-|1+t\underline{m}_n^0(z)|^{2}dH_n(t)}{|1+tE\underline{m}_n(z)|^{2}|1+t\underline{m}_n^0(z)|^{4}}|\non
&\leq&
M|E\underline{m}_n(z)-\underline{m}_n^0(z)|\Big(\int\frac{dH_n(t)}{|1+tE\underline{m}_n(z)|^{4}}\int\frac{dH_n(t)}{|1+t\underline{m}_n^0(z)|^{8}}\Big)^{1/2}\non
&\leq&
\frac{M}{nv^{5/2}}\Big(\int\frac{dH_n(t)}{|1+tE\underline{m}_n(z)|^{4}}\Big)^{1/2}.
\end{eqnarray}

As for another term in (\ref{d5}), using (\ref{d6}) successively we
have
\begin{eqnarray*}
&&|\int\frac{|1+tE\underline{m}_n(z)|^{2}-|1+t\underline{m}_n^0(z)|^{2}dH_n(t)}
{|1+tE\underline{m}_n(z)|^{4}|1+t\underline{m}_n^0(z)|^{2}}|\\
&\leq&
\int\frac{M|E\underline{m}_n(z)-\underline{m}_n^0(z)|dH_n(t)}{|1+tE\underline{m}_n(z)|^{4}|1+t\underline{m}_n^0(z)|^{2}}\non
&\leq&\int\frac{M|E\underline{m}_n(z)-\underline{m}_n^0(z)|dH_n(t)}{|1+tE\underline{m}_n(z)|^{2}|1+t\underline{m}_n^0(z)|^{4}}+\int\frac{M|E\underline{m}_n(z)-\underline{m}_n^0(z)|^2dH_n(t)}{|1+tE\underline{m}_n(z)|^{4}|1+t\underline{m}_n^0(z)|^{4}}\non
&\leq&\frac{M}{nv^{5/2}}\Big(\int\frac{dH_n(t)}{|1+tE\underline{m}_n(z)|^{4}}\Big)^{1/2}+\frac{M}{(nv^{5/2})^2}\int\frac{dH_n(t)}{|1+tE\underline{m}_n(z)|^{4}|1+t\underline{m}_n^0(z)|^{4}}\non
&\leq&\cdots\non
&\leq&(\frac{M}{nv^{5/2}}+\frac{M}{(nv^{5/2})^2}+\frac{M}{(nv^{5/2})^3}+\frac{M}{(nv^{5/2})^4})\Big(\int\frac{dH_n(t)}{|1+tE\underline{m}_n(z)|^{4}}\Big)^{1/2}\non
&&+\frac{M}{(nv^{5/2})^5}\Big(\int\frac{dH_n(t)}{|1+tE\underline{m}_n(z)|^{4}|1+t\underline{m}_n^0(z)|^{10}}\Big)^{1/2}\non
&\leq&\frac{M}{nv^{5/2}}\Big(\int\frac{dH_n(t)}{|1+tE\underline{m}_n(z)|^{4}}\Big)^{1/2}+\frac{M}{(nv^{5/2})^5v^2}\Big(\int\frac{dH_n(t)}{|1+tE\underline{m}_n(z)|^{4}}\Big)^{1/2}.
\end{eqnarray*}
where in the last step one uses (6.16) in \cite{ker}. This, together with
(\ref{d4}) and (\ref{d5}), yields
$$
\int\frac{dH_n(t)}{|1+tE\underline{m}_n(z)|^{4}}|\leq
M+\Big[\frac{M}{nv^{5/2}}+\frac{M}{(nv^{5/2})^5v^2}\Big]\Big(\int\frac{dH_n(t)}{|1+tE\underline{m}_n(z)|^{4}}\Big)^{1/2}.
$$
Solving the inequality gives
$$
\int\frac{dH_n(t)}{|1+tE\underline{m}_n(z)|^{4}}<\infty.
$$

Consider (\ref{f9}) now.  By (\ref{d12}), (\ref{f10}), (\ref{d3})
and Lemma \ref{lem1}
$$
|b_{12}(z)+z\underline{m}_n^0(z)|\leq\frac{M}{nv^{5/2}}.
$$
Applying this inequality and repeating the argument for (\ref{f11})
we may prove (\ref{f9}) and omit the details here.

\noindent
{\bf Proof of Theorem \ref{rem2}.} Write
$$
nh\Big[\frac{1}{h}\int^b_a
K(\frac{x-y}{h})dF^{c_n,H_n}(y)-f_{c_n,H_n}(x)\Big]
$$$$
=nh\Big[\int^{\frac{x-a}{h}}_{\frac{x-b}{h}}
K(y)f_{c_n,H_n}(x-yh)dy-f_{c_n,H_n}(x)\Big].
$$
By Taylor's expansion
$$
f_{c_n,H_n}(x-yh)=f_{c_n,H_n}(x)-f'_{c_n,H_n}(x)yh+f''_{c_n,H_n}(x_0)(yh)^2,
$$
where $x_0$ lies in $[x-yh,x]$. This, together with (\ref{a27}),
(\ref{g51}), (\ref{g52}) and Theorem \ref{theo1}, ensures Theorem \ref{rem2}.

\section{Appendix 3}

The aim in this section is to develop the asymptotic means and variances in Theorem \ref{theo1} and Theorem \ref{theo2}. Consider (\ref{f62}) first. Note that
\begin{eqnarray*}
(\ref{f62})
&=&-\frac{1}{2h^2\pi^2}\oint_{\mathcal{C}_1}\oint_{\mathcal{C}_2}
K'(\frac{x_1-z_1}{h})K'(\frac{x_2-z_2}{h})\label{e1}\\
&&\qquad\quad\times\Big[\ln\Big|\underline{m}_n^0(z_1)-\underline{m}_n^0(z_2)\Big|+i
arg(\underline{m}_n^0(z_1)-\underline{m}_n^0(z_2))\Big]dz_1dz_2,
\end{eqnarray*}
where the contours $\mathcal{C}_1$ and $\mathcal{C}_2$ are two
rectangles defined in (\ref{a3*}) and (\ref{g53}), respectively.

As in Section 5 of \cite{b2} one may prove that
\begin{equation}\label{f58}
\inf\limits_{z\in S,n}|\underline{m}_n^0(z)|>0,\quad
\Big|\frac{\underline{m}_n^0(z_1)-\underline{m}_n^0(z_2)}{z_1-z_2}\Big|\geq\frac{1}{2}|\underline{m}_n^0(z_1)\underline{m}_n^0(z_1)|,
\end{equation}
where $S$ is any bounded subset of $\mathbb{C}$.

To facilitate statements, denote the real parts of $z_j$ by
$u_j$,$j=1,2$. In what follows, let $n\rightarrow\infty$ first and
then $v_0\rightarrow 0$. Then, as argued in \cite{b2}, the integrals
in (\ref{e1}) involving the arg term and the vertical sides approach
zero.

Define
$$
K^{(1)}_{ri}=K_r'(\frac{x_1-z_1}{h})K_r'(\frac{x_2-z_2}{h})-K_i'(\frac{x_1-z_1}{h})K_i'(\frac{x_2-z_2}{h}),
$$
$$
K^{(2)}_{ri}=K_r'(\frac{x_1-z_1}{h})K_r'(\frac{x_2-z_2}{h})+K_i'(\frac{x_1-z_1}{h})K_i'(\frac{x_2-z_2}{h}).
$$
 Therefore it is enough to investigate the following integrals
$$
-\frac{1}{h^2\pi^2}\int_{a_l}^{a_r}\int_{a_l-\varepsilon}
^{a_r+\varepsilon}
[K^{(1)}_{ri}\ln|\underline{m}_n^0(z_1)-\underline{m}_n^0(z_2)|\non
-K^{(2)}_{ri}\ln|\underline{m}_n^0(z_1)-\overline{\underline{m}_n^0}(z_2)|]du_1du_2\non
$$
\begin{eqnarray}
&&=\frac{1}{h^2\pi^2}\int_{a_l}^{a_r}\int_{a_l-\varepsilon}
^{a_r+\varepsilon}
(K_r'(\frac{x_1-z_1}{h})K_r'(\frac{x_2-z_2}{h})\ln\Big|\frac{\underline{m}_n^0(z_1)-\overline{\underline{m}_n^0}(z_2)}{\underline{m}_n^0(z_1)-\underline{m}_n^0(z_2)}\Big|du_1du_2
\label{e4}\\
&&+\frac{1}{h^2\pi^2}\int_{a_l}^{a_r}\int_{a_l-\varepsilon}
^{a_r+\varepsilon}
(K_i'(\frac{x_1-z_1}{h})K_i'(\frac{x_2-z_2}{h})\label{e2}\\
&&\qquad\qquad\times\ln\Big|(\underline{m}_n^0(z_1)-\underline{m}_n^0(z_2))(\underline{m}_n^0(z_1)-\overline{\underline{m}_n^0}(z_2))\Big|du_1du_2\nonumber
,\end{eqnarray} where $K_r'(\frac{x-z}{h})$ and
$K_i'(\frac{x-z}{h})$, respectively, represent the real part and
imaginary part of $K'(\frac{x-z}{h})$,
$\overline{\underline{m}_n^0}(z)$ stands for the complex conjugate
of $\underline{m}_n^0(z)$.

We develop the limit of (\ref{e4}) and (\ref{e2}) below. To this
end, we list some facts below.  By (\ref{a25}) and (\ref{a26}) one
may verify that
\begin{equation}\label{a13*}
\int_{-\infty}^{+\infty} \int_{-\infty}^{+\infty}\Big|K'(u_1)
K'(u_2)\ln (u_1-u_2)^2\Big|du_1du_2<\infty.
\end{equation}
In addition, it follows from (\ref{a25}) that
$$
\ln \frac{1}{h^2}\int_{\frac{x-b}{h}}^{\frac{x-a}{h}}
K_r'(u_1)\int_{\frac{x-b-\varepsilon}{h}}
^{\frac{x-a+\varepsilon}{h}} K_r'(u_2)du_1du_2\rightarrow 0.
$$
This, together with (\ref{a13*}), implies that as
$n\rightarrow\infty$
\begin{eqnarray}
&&\frac{1}{h^2}\int_{x_1-a_r}^{x_1-a_l}\int_{x_2-a_r-\varepsilon}
^{x_2-a_l+\varepsilon} K'(\frac{u_1}{h})K'(\frac{u_2}{h})\ln
(u_1-u_2)^2du_1du_2\non
&=&\int_{\frac{x_1-a_r}{h}}^{\frac{x_1-a_l}{h}}\int_{\frac{x_2-a_r-\varepsilon}{h}}
^{\frac{x_2-a_l+\varepsilon}{h}} K'(u_1)K'(u_2)\Big[\ln
(u_1-u_2)^2-\ln \frac{1}{h^2}\Big]du_1du_2 \non
&\rightarrow&\int_{-\infty}^{+\infty}
\int_{-\infty}^{+\infty}K'(u_1) K'(u_2)\ln
(u_1-u_2)^2du_1du_2.\label{f75}
\end{eqnarray}
By (\ref{f63}) and the continuity property of $K''(u+iv_0)$ and
$K'(u+iv_0)$ in $u$ and $v_0$ it is not difficult to prove that
\begin{equation}
\label{f60}
\lim\limits_{v_0\rightarrow0}\int^{+\infty}_{-\infty}|K''(u+iv_0)|du=\int^{+\infty}_{-\infty}|K''(u)|du
\end{equation}
and
\begin{equation}\label{f72}
\lim\limits_{v_0\rightarrow0}\int^{+\infty}_{-\infty}K^{(j)}(u+iv_0)du=\int^{+\infty}_{-\infty}K^{(j)}(u)du,\ j=0,1,
\end{equation}
where $K^{(j)}$ is the $j$-th derivative of $K$.

By complex Roller's theorem
\begin{equation}\label{e5}
K_i'(\frac{x-z_1}{h})=K_i'(\frac{x-u_1}{h}+iv_0)=vK_r''(\frac{x-u}{h}+iv_1)
\end{equation}
because $K_i'(\frac{x-u_1}{h})=0$, where $v_1$ lies in $(0,v_0))$.
Thus we conclude from (\ref{f58}) and (\ref{f60}) that
$$
|\frac{1}{h}\int_{a_l}^{a_r}
(K_i'(\frac{x_1-z_1}{h})\ln\Big|(\underline{m}_n^0(z_1)-\underline{m}_n^0(z_2))(\underline{m}_n^0(z_1)-\overline{\underline{m}_n^0}(z_2))\Big|du_1|
$$
$$
\leq v_0h\ln
(v_0^{-1}h)\frac{1}{h}\int_{a}^b|K''(\frac{x-u}{h}+iv_1)|du_1\rightarrow
0,
$$
as $n\rightarrow\infty,\ v_0\rightarrow 0$, which implies that
(\ref{e2}) converges to zero.

Consider (\ref{e4}) next. We claim that for $u\in
[\frac{x-b}{h},\frac{x-a}{h}]$, as $n\rightarrow \infty$,
\begin{equation}\label{f66}
|\underline{m}_n^0(z_n)-\underline{m}(u_n)|\rightarrow 0,
\end{equation}
where $z_n=u_n-iv_0h$ with $u_n=x-uh$. Indeed, from (3.10) in
\cite{b4} we have
\begin{equation}
\label{a5}
z(\underline{m}_n^0)=-\frac{1}{\underline{m}_n^0}+c_n\int\frac{tdH_n(t)}{1+t\underline{m}_n^0},
\end{equation}
(one may also refer to Section 6.3 of \cite{ker}).
 Then, as pointed out in Lemma 1 \cite {ker}, relying on this
expression we may draw the conclusions for $\underline{m}_n^0$
similar to those in Theorem 1.1 of \cite{s1} for $\underline{m}(z)$.
Thus we have
\begin{equation}\label{f76} |\underline{m}_n^0(z_n)-\underline{m}_n^0(u_n)|\rightarrow
0.
\end{equation}
Also, the argument of Lemma 2 in \cite{ker} gives
\begin{equation}
\label{f77} |\underline{m}_n^0(u_n)-\underline{m}(u_n)|\rightarrow
0.
\end{equation}
Therefore, (\ref{f66}) is true, as claimed.

Now, as in \cite{b2}, for (\ref{e4}) write
\begin{equation}\label{e8}
\ln\Big|\frac{\underline{m}_n^0(z_1)-\overline{\underline{m}_n^0}(z_2)}{\underline{m}_n^0(z_1)-\underline{m}_n^0(z_2)}\Big|
=\frac{1}{2}\ln\Big(1+\frac{4\underline{m}_{ni}^0(z_1)\underline{m}_{ni}^0(z_2)}{|\underline{m}_n^0(z_1)-\underline{m}_n^0(z_2)|^2}\Big),
\end{equation}
where $\underline{m}_{ni}^0(z)$ denotes the imaginary part of
$\underline{m}_{n}^0(z)$. By (\ref{f58})
\begin{equation}\label{f74}
\ln\Big(1+\frac{4\underline{m}_{ni}^0(z_1)\underline{m}_{ni}^0(z_2)}{|\underline{m}_n^0(z_1)-\underline{m}_n^0(z_2)|^2}\Big)\leq
\ln\Big(1+\frac{16\underline{m}_{ni}^0(z_1)\underline{m}_{ni}^0(z_2)}{(u_1-u_2)^2|\underline{m}_n^0(z_1)\underline{m}_n^0(z_2)|^2}\Big).
\end{equation}
In view of (\ref{f58}) and Lemma \ref{lem1}
\begin{equation}\label{f73}
\sup\limits_{u_1,u_2\in
[a,b],v_1,v_2\in[v_0h,1]}\Big|\frac{\underline{m}_{ni}^0(z_1)\underline{m}_{ni}^0(z_2)}{|\underline{m}_n^0(z_1)\underline{m}_n^0(z_2)|^2}\Big|<\infty.
\end{equation}
By the generalized dominated convergence theorem we then conclude
from (\ref{f75}), (\ref{f72}), (\ref{f66}), (\ref{f74}), (\ref{f73})
that as $n\rightarrow\infty$
$$\int_{\frac{x_1-a_r}{h}}^{\frac{x_1-a_l}{h}}\int_{\frac{x_2-a_r-\varepsilon}{h}}
^{\frac{x_2-a_l+\varepsilon}{h}}K_r'(z_1)
K_r'(z_2)\Big[\ln\Big|\frac{\underline{m}_n^0(u_{n1}-iv_0h)-\overline{\underline{m}_n^0}(u_{n2}-iv_0h/2)}
{\underline{m}_n^0(u_{n1}-iv_0h)-\underline{m}_n^0(u_{n2}-iv_0h/2)}\Big|$$$$-\ln\Big|\frac{\underline{m}(u_{n1})-\overline{\underline{m}}(u_{n2})}
{\underline{m}(u_{n1})-\underline{m}(u_{n2})}\Big|\Big]du_1du_2\longrightarrow
0,
$$
where $u_{nj}=x_j-u_jh$, $j=1,2$. In addition, it follows from
(\ref{f75}), (\ref{f72}), and inequalities similar to (\ref{f74})
and (\ref{f73}) that as $n\rightarrow\infty$ and then
$v_0\rightarrow0$
\begin{eqnarray*}
\int_{\frac{x_1-a_r}{h}}^{\frac{x_1-a_l}{h}}\int_{\frac{x_2-a_r-\varepsilon}{h}}
^{\frac{x_2-a_l+\varepsilon}{h}}(K_r'(z_1) K_r'(z_2)-K_r'(u_1)
K_r'(u_2))\ln\Big|\frac{\underline{m}(u_{n1})-\overline{\underline{m}}(u_{n2})}
{\underline{m}(u_{n1})-\underline{m}(u_{n2})}\Big|du_1du_2\\
\rightarrow
0.
\end{eqnarray*}
Therefore (\ref{e4}) can be reduced to the following
\begin{eqnarray}
\qquad \int_{\frac{x_1-a_r}{h}}^{\frac{x_1-a_l}{h}}\int_{\frac{x_2-a_r-\varepsilon}{h}}
^{\frac{x_2-a_l+\varepsilon}{h}}K'(u_1)
K'(u_2)\ln\Big|\frac{\underline{m}(u_{n1})-\overline{\underline{m}}(u_{n2})}
{\underline{m}(u_{n1})-\underline{m}(u_{n2})}\Big|du_1du_2+o(1),\label{f71}
\end{eqnarray}
which turns to be
$$\frac{1}{h^2}\int_{x_1-a_r}^{x_1-a_l}\int_{x_2-a_r-\varepsilon}
^{x_2-a_l+\varepsilon}
K'(\frac{u_1}{h})K'(\frac{u_2}{h})\ln\Big|\frac{\underline{m}(x_1-u_1)-\overline{\underline{m}}(x_2-u_2)}
{\underline{m}(x_1-u_1)-\underline{m}(x_2-u_2)}\Big|du_1du_2+o(1).
$$

To handle (\ref{f71}), we need two more lemmas:
\begin{lemma}\label{lem4} Suppose that the function $g(x_1,x_2)$ is continuous in $x_1$ and $x_2$,
\begin{equation}\label{f64}
\int_{x_1-a_r}^{x_1-a_l}\int_{x_2-a_r}^{x_2-a_l}|g(x_1-u_1,x_2-u_2)|du_1du_2<\infty
\end{equation}
and
\begin{equation}\label{f64}
\int_{x_1-a_r}^{x_1-a_l}|g(x_1-u_1,x_2)|du_1<\infty,\quad
\int_{x_2-a_r}^{x_2-a_l}|g(x_1,x_2-u_2)|du_2<\infty.
\end{equation} Then, as $n\rightarrow\infty$
\begin{equation}\label{e7}
\frac{1}{h^2}\int_{x_1-a_r}^{x_1-a_l}\int_{x_2-a_r-\varepsilon}
^{x_2-a_l+\varepsilon}
K'(\frac{u_1}{h})K'(\frac{u_2}{h})g(x_1-u_1,x_2-u_2)du_1du_2\rightarrow
0,
\end{equation}
where $x_1\neq a_l, a_r$ and $x_2\neq a_l, a_r$.
\end{lemma}

 Define the sets $G_1=(|u_1|\leq\delta_1)\cap (|u_2|>\delta_2)$, $G_2=(|u_1|>\delta_1)\cap (|u_2|\leq\delta_2)$ and $G_3=(|u_1|>\delta_1)\cap (|u_2|>\delta_2)$. Splitting the region of integration into the union of the sets $(|u_1|\leq\delta_1)\cap (|u_2|\leq\delta_2)$, $G_1$, $G_2$ and $G_3$ gives
$$\Big|\frac{1}{h^2}\int_{x_1-a_r}^{x_1-a_l}\int_{x_2-a_r-\varepsilon}
^{x_2-a_l+\varepsilon}
K'(\frac{u_1}{h})K'(\frac{u_2}{h})\Big[g(x_1-u_1,x_2-u_2)-g(x_1,x_2)\Big]du_1du_2\Big|$$
\begin{eqnarray}
\leq I_1+I_2+I_3+I_4+I_5\label{f70},
\end{eqnarray}
where
$$I_1= \sup\limits_{|u_1|\leq\delta_1,|u_2|\leq\delta_2}\Big|g(x_1-u_1,x_2-u_2)-g(x_1,x_2)\Big|\int_{-\infty}^{+\infty}
|K'(u)|du\Big|^2,
$$
$$I_2= |g(x_1,x_2)|\Big|\frac{1}{h^2}\int_{x_1-a_r}^{x_1-a_l}\int_{x_2-a_r-\varepsilon}
^{x_2-a_l+\varepsilon}I(G_1\cup G_2\cup G_3)
K'(\frac{u_1}{h})K'(\frac{u_2}{h})du_1du_2\Big|,
$$$$I_3=\Big|\frac{1}{h^2}\int_{x_1-a_r}^{x_1-a_l}\int_{x_2-a_r-\varepsilon}
^{x_2-a_l+\varepsilon}I(G_1)
K'(\frac{u_1}{h})K'(\frac{u_2}{h})g(x_1-u_1,x_2-u_2)du_1du_2\Big|,
$$
$$I_4=\Big|\frac{1}{h^2}\int_{x_1-a_r}^{x_1-a_l}\int_{x_2-a_r-\varepsilon}
^{x_2-a_l+\varepsilon}I(G_2)
K'(\frac{u_1}{h})K'(\frac{u_2}{h})g(x_1-u_1,x_2-u_2)du_1du_2\Big|
$$
and
$$I_5=\Big|\int_{x_1-a_r}^{x_1-a_l}\int_{x_2-a_r-\varepsilon}
^{x_2-a_l+\varepsilon}I(G_3)
\frac{u_1u_2}{h^2}K'(\frac{u_1}{h})K'(\frac{u_2}{h})\frac{g(x_1-u_1,x_2-u_2)}{u_1u_2}du_1du_2\Big|.
$$
Evidently, $I_1\rightarrow 0$ due to the continuity property of $g(x_1,x_2)$ when $\delta_1$ and $\delta_2$ converge to zero. As $n\rightarrow\infty$, for $I_2$ we have
$$
I_2\leq M|g(x_1,x_2)|\int\limits_{|u|>\delta/h}|
K'(u)|du\int_{-\infty}^{+\infty}
|K'(u)|du\rightarrow 0,
$$
and for $I_5$ by (\ref{f64}) we obtain
\begin{eqnarray*}
&&I_5\leq\frac{1}{\delta_1\delta_2}\sup\limits_{|u_1|>\delta_1/h}|u_1K'(u_1)|\sup\limits_{|u_2|>\delta_2/h}|u_2K'(u_2)|
\\
&&\qquad\qquad\times\int_{x_1-a_r}^{x_1-a_l}\int_{x_2-a_r-\varepsilon}
^{x_2-a_l+\varepsilon}|g(x_1-u_1,x_2-u_2)|du_1du_2\rightarrow 0.
\end{eqnarray*}
Consider $I_3$. Similar to $I_5$,
$$
I_3\leq\frac{1}{\delta_2}\sup\limits_{|u_2|>\delta_2/h}|u_2K'(u_2)|\int_{|u_1|\leq\delta_1/h}\int_{x_2-a_r-\varepsilon}
^{x_2-a_l+\varepsilon}|K'(u_1)g(x_1-u_1h,x_2-u_2)|du_1du_2.
$$
While, as $n\rightarrow\infty$ and then $\delta_1\rightarrow 0$, by the dominated convergence theorem
$$
\frac{1}{h}\int_{|u_1|\leq\delta_1}|K'(\frac{u_1}{h})|\int_{x_2-a_r-\varepsilon}
^{x_2-a_l+\varepsilon}|(g(x_1-u_1,x_2-u_2)-g(x_1,x_2-u_2))|du_1du_2\rightarrow 0.
$$
From (\ref{f64}) we then see that $I_3\rightarrow 0$. One may
similarly prove that $I_4$ converges to zero as well. We summarize
the above that (\ref{f70}) converges to zero as $n\rightarrow\infty$
first and then both $\delta_1\rightarrow0$ and
$\delta_2\rightarrow0$. In addition, apparently,
\begin{eqnarray}
g(x_1,x_2)\frac{1}{h}\int_{x_1-a_r}^{x_1-a_l}
K'(\frac{u}{h})du&=&g(x_1,x_2)\int_{\frac{x_1-a_l}{h}}^{\frac{x_1-a_r}{h}}
K'(u)du\label{f78}\\
&=&g(x_1,x_2)K(u)\Big|_{\frac{x_1-a_l}{h}}^{\frac{x_1-a_r}{h}}\rightarrow 0. \nonumber
\end{eqnarray}
Thus (\ref{e7}) is proved.

The next lemma extends (1.6) in \cite{s1}, which now includes the
boundary points of $F^{c,H}(x)$ under some extra conditions.
\begin{lemma} \label{lem5} Suppose that the support of $F^{c,H}(x)$ is $[a,b]$
with $a>0$ and $b$ finite. Then $\underline{m}(x)$ is the unique
solution to the equation
\begin{equation}\label{f68*}
x=-\frac{1}{\underline{m}(x)}+c\int\frac{\lambda
dH(\lambda)}{1+\lambda\underline{m}(x)},
\end{equation}
where $\lim\limits_{z\rightarrow
x}\underline{m}(z)=\underline{m}(x)$.
\end{lemma}
\begin{proof} When $u$, the real part of $z$, is bounded, we have
$$
\Im (\underline{m}(z))\geq \frac{v}{M+v^2}.
$$
It follows that
\begin{equation}\label{f69*}
\frac{v}{\Im (\underline{m}(z))}\leq M+v^2.
\end{equation}
 Considering the imaginary parts of both
sides of the equality (\ref{f51}) yields
$$
v=\frac{\Im(\underline{m}(z))}{|\underline{m}(z)|^2}-c\Im(\underline{m}(z))\int\frac{t^2dH(t)}{|1+t\underline{m}(z)|^2},
$$
which, together with (\ref{f58}) and (\ref{f69*}), implies
\begin{equation}\label{f70}
\sup\limits_{z\in S}\int\frac{t^2dH(t)}{|1+t\underline{m}(z)|^2}\leq
M.
\end{equation}
Taking $z\rightarrow x$ in (\ref{f51}) and using (1.5) in \cite{s1}
we then see that (\ref{f68*}) is true. The uniqueness of
$\underline{m}(x)$ is from continuity of $\underline{m}(x)$ and the
uniqueness of $\underline{m}(x)$ when $\Im\underline{m}(x)\neq 0$
given in \cite{s1}.
\end{proof}

We are now in a position to apply Lemma \ref{lem4} to (\ref{f71}).
It follows from Lemma \ref{lem5} that $\underline{m}(x_1)\neq
\underline{m}(x_2)$ and $\underline{m}(x_1)\neq
\overline{\underline{m}}(x_2)$ whenever $x_1\neq x_2$. Also, note
(5.1) in \cite{b2}. Therefore
$g(x_1,x_2)=\ln\Big|\frac{\underline{m}(x_1)-\overline{\underline{m}}(x_2)}
{\underline{m}(x_1)-\underline{m}(x_2)}\Big|$ is continuous in $x_1$
and $x_2$. Furthermore, it is straightforward to show that $ \ln
\Big(1+\frac{M}{(x_1-x_2)-(u_1-u_2)} \Big) $ for $u_1,u_2\in
[a_l-\varepsilon,a_r+\varepsilon]$ is Lebesgue integrable and $\ln
\Big(1+\frac{M}{(x_1-x_2)-(u_1)} \Big)$ for $u_2\in
[a_l-\varepsilon,a_r+\varepsilon]$ is Lebesgue integrable. Thus, in
view of inequalities similar to (\ref{e8})-(\ref{f73}) and applying
(\ref{e7}) we have
\begin{equation}\label{f67}
\frac{1}{h^2}\int_{x_1-a_r}^{x_1-a_l}\int_{x_2-a_r-\varepsilon}
^{x_2-a_l+\varepsilon}
K'(\frac{u_1}{h})K'(\frac{u_2}{h})\ln\Big|\frac{\underline{m}(x_1-u_1)-\overline{\underline{m}}(x_2-u_2)}
{\underline{m}(x_1-u_1)-\underline{m}(x_2-u_2)}\Big|du_1du_2\rightarrow
0,
\end{equation}
which is the limit of (\ref{e4}) due to (\ref{f71}) when $x_1\neq
x_2$.

 When $x_1=x_2=x$ taking
$g(x_1,x_2)=\ln\Big|\underline{m}(x)-\overline{\underline{m}}(x)\Big|$
and applying (\ref{e7}) we obtain
\begin{equation}\label{f67}
\frac{1}{h^2}\int_{x-a_r}^{x-a_l}\int_{x-a_r-\varepsilon}
^{x-a_l+\varepsilon}
K'(\frac{u_1}{h})K'(\frac{u_2}{h})\ln\Big|\underline{m}(x-u_1)-\overline{\underline{m}}(x-u_2)\Big|du_1du_2\rightarrow 0.
\end{equation}
Here we keep in mind that the boundary points are not considered when
investigating the case $x_1=x_2=x$. Consider next
\begin{equation}\label{a12}
\frac{1}{h^2}\int_{x-b}^{x-a}
K'(\frac{u_1}{h})\int_{x-b-\varepsilon} ^{x-a+\varepsilon}
K'(\frac{u_2}{h})\ln\Big|\underline{m}(x-u_1)-\underline{m}(x-u_2)\Big|du_1du_2.
\end{equation}
By complex Roller's theorem we have
\begin{eqnarray}
&&\ln\Big|\underline{m}(x-u_1)-\underline{m}(x-u_2)\Big|\non&=&\frac{1}{2}\ln
\Big((u_1-u_2)^2[|\underline{m}'_r(x-u_3)|^2+|\underline{m}'_i(x-u_4)|^2]\Big)\non
&=&\frac{1}{2}\ln (u_1-u_2)^2
+\frac{1}{2}g_{ri}(x-u_1,x-u_2),\label{f65}
\end{eqnarray}
where $g_{r}(x-u_1,x-u_2)=\ln
\Big(|\underline{m}'_r(t_1(x-u_1)+(1-t_1)(x-u_2))|^2+|\underline{m}'_i(t_2(x-u_1)+(1-t_2)(x-u_2))|^2\Big),$
$u_3=t_1u_1+(1-t_1)u_2,\ u_4=t_2u_1+(1-t_2)u_2$ and $t_1,t_2\in
(0,1)$. It follows from inequalities for $\underline{m}(x)$ similar
to (\ref{f58}) that
$$
\Big|\int_{x-b}^{x-a} \int_{x-b-\varepsilon} ^{x-a+\varepsilon}
\ln\Big|\underline{m}(x-u_1)-\underline{m}(x-u_2)\Big|du_1du_2\Big|<\infty.
$$
This, together with (\ref{f65}), ensures that
$$
\Big|\int_{x-b}^{x-a} \int_{x-b-\varepsilon} ^{x-a+\varepsilon}
g_{r}(x-u_1,x-u_2)du_1du_2\Big|<\infty.
$$
Similarly, one may verify the remaining conditions in Lemma
\ref{lem4}. Therefore, using Lemma \ref{lem4} with
$g(x_1,x_2)=\ln|\underline{m}'(x)|^2$ gives
\begin{equation}\label{f68}
\frac{1}{h^2}\int_{x-b}^{x-a}
K'(\frac{u_1}{h})\int_{x-b-\varepsilon} ^{x-a+\varepsilon}
K'(\frac{u_2}{h})g_{r}(x-u_1,x-u_2)du_1du_2\rightarrow 0.
\end{equation}
We then conclude from (\ref{f65}), (\ref{f68}) and (\ref{f75}) that
$$
(\ref{a12})=\frac{1}{2}\frac{1}{h^2}\int_{x-b}^{x-a}
K'(\frac{u_1}{h})\int_{x-b-\varepsilon} ^{x-a+\varepsilon}
K'(\frac{u_2}{h})\ln (u_1-u_2)^2du_1du_2+o(1)
$$
\begin{equation}\label{a13}
\rightarrow\frac{1}{2}\int_{-\infty}^{+\infty}
\int_{-\infty}^{+\infty}K'(u_1) K'(u_2)\ln (u_1-u_2)^2du_1du_2.
\end{equation}
which is the opposite number of the limit of (\ref{e4}) due to (\ref{f67}) and (\ref{f71})
when $x_1=x_2$.

\noindent
{\bf Limit of (\ref{f50})}.  From an expression similar to (\ref{f51})
we obtain
$$
\frac{d}{dz}\underline{m}_n^0(z)=\frac{(\underline{m}_n^0(z))^2}{1-c\int\frac{t^2(\underline{m}_n^0(z))^2}{(1+t\underline{m}_n^0(z))^2}dH_n(t)}.
$$
It follows that (\ref{f50}) becomes
$$
\frac{1}{4\pi i}\oint K(\frac{x-z}{h})\frac{d}{dz}\log
\Big[1-c\int\frac{t^2(\underline{m}_n^0(z))^2}{(1+t\underline{m}_n^0(z))^2}dH_n(t)\Big]dz
$$
\begin{equation}\label{f52}
=\frac{1}{4\pi hi}\oint K'(\frac{x-z}{h})\log
\Big[1-c\int\frac{t^2(\underline{m}_n^0(z))^2}{(1+t\underline{m}_n^0(z))^2}dH_n(t)\Big]dz
\end{equation}
As in the inequality above (6.37) in \cite{ker} and (3.21) in \cite{b4} one may prove that
\begin{equation}\label{f53}
|1-c\int\frac{t^2(\underline{m}_n^0(z))^2}{(1+t\underline{m}_n^0(z))^2}dH_n(t)|\geq
Mv.
\end{equation}
This implies that the integrals on the two vertical lines in
(\ref{f52}) are bounded by $Mv\log v^{-1}$, which converges to zero as $v\to 0$.
The integrals on the two horizontal lines are equal to
\begin{eqnarray}
&&\frac{1}{2\pi h}\int K_i'(\frac{x-z}{h})\log \Big
|1-c\int\frac{t^2(\underline{m}_n^0(z))^2}{(1+t\underline{m}_n^0(z))^2}dH_n(t)\Big
|du\label{f54}
\\ &+& \frac{1}{2\pi h}\int K_r'(\frac{x-z}{h})\arg
\Big
[1-c\int\frac{t^2(\underline{m}_n^0(z))^2}{(1+t\underline{m}_n^0(z))^2}dH_n(t)\Big
]du.\label{f55}
\end{eqnarray}
By (2.19) in \cite{b2}, (\ref{e5}) and (\ref{f53}) we see that
(\ref{f54}) is bounded by $Mv\log v^{-1}$, converging to zero. It follows from (\ref{f66}) and Lemma \ref{lem3}
that
$$
\int\frac{t^2(\underline{m}_n^0(z_n))^2}{(1+t\underline{m}_n^0(z_n))^2}dH_n(t)-\int\frac{t^2(\underline{m}(u_n))^2}{(1+t\underline{m}(u_n))^2}dH_n(t)\rightarrow0.
$$
We also claim that
\begin{equation}
\label{f69}
\int\frac{t^2(\underline{m}(u_n))^2}{(1+t\underline{m}(u_n))^2}dH_n(t)-\int\frac{t^2(\underline{m}(u_n))^2}{(1+t\underline{m}(u_n))^2}dH(t)\rightarrow 0.
\end{equation}
To see this, introduce random variables $T_n$ having distribution $H_n(t)$ and $T$ having distribution $H(t)$. Then $T_n\stackrel{D}\longrightarrow T$. Also $T_n$ and $T$ are both bounded. Consequently by Lemma \ref{lem3}
$$
E\Big|\frac{T_n^2}{(1+T_n\underline{m}(u_n))^2}-\frac{T^2}{(1+T\underline{m}(u_n))^2}\Big|
$$$$\leq \Big(E|\frac{T_n}{1+T_n\underline{m}(u_n)}-\frac{T}{1+T\underline{m}(u_n)}|^2)E|\frac{T_n}{1+T_n\underline{m}(u_n)}+\frac{T}{1+T\underline{m}(u_n)}|^2\Big)^{1/2}
$$$$\leq M\Big(E|\frac{1}{1+T_n\underline{m}(u_n)}|^6E|\frac{1}{1+T\underline{m}(u_n)}|^6E|T_n-T|^3\Big)^{1/3}
$$
converging to zero. Thus (\ref{f69}) is true, as claimed. We then conclude from the dominated convergence theorem that
$$
\int K_r'(z)\arg \Big
[1-c\int\frac{t^2(\underline{m}_n^0(z_n))^2}{(1+t\underline{m}_n^0(z_n))^2}dH_n(t)\Big]
$$$$-\arg \Big
[1-c\int\frac{t^2(\underline{m}(u_n))^2}{(1+t\underline{m}(u_n))^2}dH(t)\Big
]du\rightarrow0.
$$
Moreover, by (\ref{f72}) we obtain
$$
\int (K_r'(z)-K_r'(u))
\arg \Big
[1-c\int\frac{t^2(\underline{m}(u_n))^2}{(1+t\underline{m}(u_n))^2}dH(t)\Big
]du\rightarrow0.
$$
By (\ref{f78}) and Theorem 1A in \cite{p1} (replacing $K(x)$ there by $K'(x)$) we see that
$$
\int K_r'(u))
\arg \Big
[1-c\int\frac{t^2(\underline{m}(u_n))^2}{(1+t\underline{m}(u_n))^2}dH(t)\Big
]du\rightarrow0.
$$
Summarizing the above yields that (\ref{f50}) converges to zero.

\noindent
{\bf Limits of (\ref{g43}) and (\ref{g44})}. Repeating the argument leading to (\ref{f71}) yields (\ref{g43}) becomes
\begin{equation}\label{g45}\frac{1}{h^2\ln \frac{1}{h}}\int_{x_1-a_r}^{x_1-a_l}\int_{x_2-a_r-\varepsilon}
^{x_2-a_l+\varepsilon}
K(\frac{u_1}{h})K(\frac{u_2}{h})\ln\Big|\frac{\underline{m}(x_1-u_1)-\overline{\underline{m}}(x_2-u_2)}
{\underline{m}(x_1-u_1)-\underline{m}(x_2-u_2)}\Big|du_1du_2+o(1).
\end{equation}
The argument of (\ref{f70}) in Lemma \ref{lem4} indeed also, together with (\ref{a26}),  gives
\begin{equation}\label{g46}
\frac{1}{h^2}\int_{x_1-a_r}^{x_1-a_l}\int_{x_2-a_r-\varepsilon}
^{x_2-a_l+\varepsilon}
K(\frac{u_1}{h})K(\frac{u_2}{h})g(x_1-u_1,x_2-u_2)du_1du_2-g(x_1,x_2)\rightarrow 0.
\end{equation}
This ensures that (\ref{g45}) converges to zero when $x_1\neq x_2$. When $x_1= x_2=x$, by (\ref{g46}) we have
$$
\frac{1}{h^2\ln \frac{1}{h}}\int_{x_1-a_r}^{x_1-a_l}\int_{x_2-a_r-\varepsilon}
^{x_2-a_l+\varepsilon}
K(\frac{u_1}{h})K(\frac{u_2}{h})\ln\Big|\underline{m}(x_1-u_1)-\overline{\underline{m}}(x_2-u_2)\Big|du_1du_2\rightarrow 0.
$$
Applying (\ref{g46}) and replacing $K'(x)$ in (\ref{f75}), (\ref{f65}), (\ref{f68}) and (\ref{a13}) by $K(x)$, we can prove that
$$
-\frac{1}{h^2\ln \frac{1}{h}}\int_{x_1-a_r}^{x_1-a_l}\int_{x_2-a_r-\varepsilon}
^{x_2-a_l+\varepsilon}
K(\frac{u_1}{h})K(\frac{u_2}{h})\ln\Big|\underline{m}(x_1-u_1)-\underline{m}(x_2-u_2)\Big|du_1du_2\rightarrow 1.
$$

Checking on the argument of (\ref{f50}) and replacing $K'(x)$ there
with $K(x)$, along with(\ref{g47}), we have
$$
(\ref{g44})\rightarrow 0.
$$

\end{document}